\documentclass{article}
\usepackage[utf8]{inputenc}
\usepackage{amsfonts}
\usepackage{amssymb}
\usepackage{amsmath}
\usepackage{amsthm}
\usepackage{commath}
\usepackage{xcolor}
\usepackage{bbold}
\usepackage{mathrsfs}
\usepackage{enumitem}
\usepackage{graphicx}
\usepackage{cancel}
\usepackage[noabbrev]{cleveref}
\usepackage{tikz-cd}

\newcommand{\R}{\mathbb{R}}
\newcommand{\T}{\mathbb{T}}
\newcommand{\Z}{\mathbb{Z}}
\newcommand{\C}{\mathbb{C}}

\newcommand{\Id}{\textup{Id}}
\newcommand{\A}{\mathcal{A}}

\newcommand{\dV}{\,d\mathcal{V}}
\newcommand{\wt}[1]{\widetilde{#1}}

\newcommand{\grad}{\textup{grad}\,}
\renewcommand{\del}{\partial}

\newcommand{\M}{\mathcal{M}}
\newcommand{\N}{\mathbb{N}}

\renewcommand{\L}{\mathcal{L}}
\newcommand{\loc}{\textit{loc}}
\newcommand{\F}{\mathcal{F}}
\newcommand{\delslash}{\cancel{\del}}
\let\pp\S
\renewcommand{\S}{\mathcal{S}}
\newcommand{\X}{\mathcal{X}}
\newcommand{\V}{\mathcal{V}}

\newcommand{\W}{\mathcal{W}}

\newtheorem{theorem}{Theorem}
\newtheorem{definition}[theorem]{Definition}
\newtheorem{lemma}[theorem]{Lemma}
\newtheorem{corollary}[theorem]{Corollary}
\newtheorem{remark}[theorem]{Remark}
\newtheorem{proposition}[theorem]{Proposition}
\newtheorem{example}[theorem]{Example}

\tikzset{every picture/.style={line width=0.75pt}} 

\numberwithin{theorem}{section}
\title{Generalizing symplectic topology from 1 to 2 dimensions}
\author{Ronen Brilleslijper and Oliver Fabert}
\date{}

\begin{document}

\maketitle


\begin{abstract}
    In symplectic topology one uses elliptic methods to prove rigidity results about symplectic manifolds and solutions of Hamiltonian equations on them, where the most basic example is given by geodesics on Riemannian manifolds. Harmonic maps from surfaces are the natural 2-dimensional generalizations of geodesics. In this paper, we give the corresponding generalization of symplectic manifolds and Hamiltonian equations, leading to a class of partial differential equations that share properties similar to Hamiltonian (ordinary) differential equations. Two rigidity results are discussed: a non-squeezing theorem and a version of the cuplength result for quadratic Hamiltonians on cotangent bundles. The proof of the latter uses a generalization of Floer curves, for which the necessary Fredholm and compactness results will be proven.
\end{abstract}

\tableofcontents

\section*{Introduction}
In topology, one uses Morse theory to prove lower bounds for the number of $0$-dimensional objects, namely critical points of smooth functions on smooth manifolds. In symplectic topology, one uses Floer theory, an infinite-dimensional generalization of Morse theory, to prove lower bounds for the number of $1$-dimensional objects, like time-periodic orbits of (Hamiltonian) functions on symplectic manifolds. Despite the infinite-dimensionality one finds that Floer theory for Hamiltonian orbits has properties very similar to Morse theory for critical points: in the case of closed symplectic manifolds the lower bound is given by the sum of the Betti numbers. This naturally leads to the question whether the results of (symplectic) topology can be generalized to prove topological bounds for the number of $2$-dimensional or even higher dimensional objects in a finite-dimensional manifold.

We will give an affirmative answer for the case of $2$ dimensions. There are two natural candidates for the generalization of symplectic geometry from $1$ to $2$ dimensions: holomorphic symplectic geometry and polysymplectic (or multi-/$2$-symplectic) geometry.

\emph{Holomorphic symplectic geometry} is a complex version of (real) symplectic geometry in the sense that one considers a complex manifold with a holomorphic symplectic $2$-form; in the similar way one considers complex-valued Hamiltonian functions which are assumed to be holomorphic. It is already known, see e.g. the work of Doan-Rezchikov and Kontsevich-Soibelman on holomorphic Floer theory, that the elliptic methods for proving rigidity results in symplectic geometry have a natural analogue in holomorphic symplectic geometry: the role of $2$-dimensional holomorphic curves in symplectic geometry is taken over by $3$-dimensional Fueter maps in holomorphic symplectic geometry (see \cite{doan2022holomorphic,kontsevich2024holomorphic}). However, it is not possible to formulate a non-linear Laplace equation, the natural higher-dimensional generalization of Newton’s equation, as a holomorphic Hamiltonian system. In particular, while geodesics are the basic examples of solutions of Hamilton’s equations in the real case, there is no holomorphic Hamiltonian system whose solution leads to harmonic maps, the natural $2$-dimensional generalization of geodesics. Moreover, there is no canonical real-valued action functional defined, but rather choices have to be made in order to find an action functional whose gradient lines may be studied.

\emph{Polysymplectic geometry} (and its multisymplectic and $k$-symplectic variants) is a geometric framework that has been developed to generalize symplectic geometry from classical mechanics to classical field theory, that is, to the case where the time coordinate is replaced by multiple (space-)time coordinates. In contrast to holomorphic symplectic geometry, it provides a framework in which non-linear Laplace equations (in particular harmonic maps) can be studied as generalized Hamiltonian systems. Here we would like to stress that the underlying Hamiltonian function is always real-valued. The central geometric object is the polysymplectic form, which is an $\R^2$-valued real $2$-form. Surprising at first glance, there does not exist anything like polysymplectic topology; in particular, there are no analogues of Gromov’s non-squeezing theorem as well as of Floer’s proof of the Arnold conjecture. Indeed, as we have already discussed in our previous papers \cite{regularizedpolysympl,hyperkahlerrookworst}, it is not possible to generalize the elliptic methods of symplectic geometry to the standard polysymplectic setting: The corresponding generalization of Hamilton’s equations, called De Donder-Weyl equations, are not elliptic due to the presence of an infinite-dimensional kernel.

In this paper we discuss a geometric framework, that we call \emph{complex-regularized polysymplectic geometry}, which is designed in order to combine the benefits of both candidates. It is possible to formulate non-linear Laplace equations, in particular harmonic maps, as Hamiltonian systems in our framework just like in polysymplectic geometry. At the same time, the elliptic methods of symplectic geometry generalize to our framework along the lines of holomorphic symplectic geometry. While we give an intrinsic definition, we prove an analogue of the Darboux theorem stating that our complex-regularized polysymplectic manifolds are locally standard. Since the polysymplectic structure builds on the holomorphic symplectic one, we can show that the holomorphic Hamiltonian functions and holomorphic Lagrangians of holomorphic symplectic geometry become part of our polysymplectic framework alongside the real Hamiltonian functions used to describe non-linear Laplace equations and harmonic maps. However, it turns out that this framework still shares rigidity properties similar to symplectic geometry. We discuss a non-squeezing result as well as a version of the cuplength result for quadratic Hamiltonians, which in particular provides a lower bound to the number of harmonic maps. Since it is apparent that both the polysymplectic as well as the holomorphic symplectic picture simultaneously play a central role in the rigidity result, we hope that the reader is convinced that our new geometric framework is the natural candidate for the generalization of symplectic geometry from $1$ to $2$ dimensions. At the end of \cref{part:formalism} of the paper, we give examples of problems from both holomorphic symplectic as well as polysymplectic geometry that naturally lead to questions in the other. Both problems can be studied within our complex-regularized polysymplectic framework.

\subsection*{Overview of the article}
This articles contains both an expository part to introduce the framework in which we work, and a technical part providing proofs of the rigidity results. We start by introducing polysymplectic geometry, with the non-linear Laplace equation as the main example of a Hamiltonian system. In \cref{sec:complexBridges} we discuss a different Hamiltonian formulation of the Laplace equation using Dirac operators, leading to the definition of complex-regularized polysymplectic manifolds. Consequently, the correspondence with holomorphic symplectic geometry is explained in \cref{sec:holomsympl} and we find a Darboux theorem for complex-regularized polysymplectic manifolds. \Cref{sec:Lagrange} describes the Lagrange formalism corresponding to this new Hamiltonian framework. To finish the expository part of the paper, \cref{sec:applications} describes two problems that highlight the interplay between holomorphic symplectic and polysymplectic geometry. The first problem is obstruction theory to holomorphic Lagrangian embeddings, a problem from holomorphic symplectic geometry, which may be answered using the study of harmonic maps. The second one is the existence of harmonic maps with specified boundary. This problem in polysymplectic geometry may be formulated as a Hamiltonian PDE, whose boundary condition is given by a complex Lagrangian in a holomorphic symplectic manifold. Both problems discussed can be studied within the complex-regularized polysymplectic framework.

Subsequently, we turn to the rigidity results. In \cref{sec:nonsqueezing}, a non-squeezing result is discussed which in particular implies that morphisms between complex-regularized polysymplectic manifolds are stable under taking $C^0$-limits. \Cref{sec:actionfunctional} discusses the action functional and formulates a version of the cuplength result for quadratic Hamiltonians, which is proven in \cref{sec:proofArnold,sec:Fredholm,sec:Compactness} by analysing a moduli space of Floer maps. While \cref{sec:proofArnold} proves the necessary $C^0$-bounds and discusses the general proof strategy, \cref{sec:Fredholm,sec:Compactness} cover the Fredholm and compactness theories respectively. 

\textbf{Acknowledgment.} The authors would like to thank Gabriele Benedetti, Luca Asselle and Alberto Abbondandolo for helpful discussions.

\part{Introducing the formalism}\label{part:formalism}
\section{Polysymplectic geometry}\label{sec:DDW}
In its most basic form, symplectic geometry studies maps $q:M\to\R^n$, where $M$ is either $\R$ or $S^1$, satisfying Newton's equation
\begin{align}\label{eq:Newton}
    -\del_t^2 q = \frac{dV}{dq},
\end{align}
for some potential $V:\R^n\to\R$. The Hamiltonian form of this equation can be described using a non-degenerate 2-form, called the symplectic form, that maps Hamiltonians on phase space $\R^{2n}$ to vector fields. The flow of these Hamiltonian vector fields then yield  first-order ODEs that generalize \cref{eq:Newton}. For example, when the Hamiltonian is chosen to be purely quadratic in the momentum variable, the flow of the Hamiltonian vector field describes the geodesic equation. Since all the geometric information is stored in the symplectic form, this Hamiltonian formalism can be extended to any manifold admitting a non-degenerate 2-form. 

As mentioned in the introduction, the polysymplectic formalism, in its most basic form, studies the non-linear Laplace equation
\begin{align}\label{eq:Laplace}
    -\Delta q := -\sum_{l=1}^d \del_{t_l}^2 q = \frac{dV}{dq},
\end{align}
where in this case $M=\R^d$. When $M$ is a surface, one would hope that the corresponding Hamiltonian formalism would be capable of describing harmonic maps, just like for $d=1$ it can describe the geodesic equation.

From here on out $d=2$. The standard way of transforming \cref{eq:Laplace} into a first-order system is by introducing  momentum variables for the different 'time'-directions. This yields the so-called De Donder-Weyl equations
\begin{align}\begin{split}\label{eq:DDW}
    -\del_1 p_1 - \del_2 p_2 &= \del_q H\\
    \del_1 q &= \del_{p_1}H \\
    \del_2 q &= \del_{p_2} H,
    \end{split}
\end{align}
where $\del_l:=\del_{t_l}$. Note that for $H(q,p_1,p_2)=\frac{1}{2}p_1^2+\frac{1}{2}p_2^2 +V(q)$, indeed \cref{eq:DDW} reduces to the nonlinear Laplace \cref{eq:Laplace}. We will denote by $Z$ the map $Z=(q,p_1,p_2):\R^2\to \R^3$.

The De Donder-Weyl equations form the basis of the different extensions of symplectic geometry to formalisms that can deal with PDEs. However, for the purpose of extending elliptic methods, they are unsuitable. The most apparent problem is that even though the Laplace equation is elliptic, \cref{eq:DDW} is not. As pointed out in previous articles (e.g. \cite{bridgesTEA}), the differential operator that defines \cref{eq:DDW} has an infinite-dimensional kernel. This can be easily seen as follows. Take $q$ to be a constant and $p_1=\del_2\psi$, $p_2=-\del_1\psi$ for any function $\psi:\R^2\to\R$. Then $-\del_1p_1-\del_2p_2=0$ and also $\del_1q=\del_2q=0$, so that indeed $(q,p_1,p_2)$ satisfies \cref{eq:DDW} for $H\equiv0$. This fact makes the extension of elliptic methods to the polysymplectic setting impossible. We refer to our previous articles \cite{regularizedpolysympl,hyperkahlerrookworst} for more explanation on why this forms a problem in the specific context of Floer theory. Fortunately, it is possible to formulate a different Hamiltonian version of the Laplace \cref{eq:Laplace} using Dirac operators. In order to do so, we need to have complex structures on the spaces we are working with. These equations will lead to the notion of complex-regularized polysymplectic geometry. Before diving into this formalism in \cref{sec:complexBridges}, we discuss the geometric picture behind \cref{eq:DDW}.

\begin{remark}
    In this article we will focus on the case where the domain $M$ has a trivial tangent bundle, meaning $M=\R^2$ or $M=\T^2$. If $M$ is any other surface, the polysymplectic formalism doesn't suffice and we have to use the more complicated multisymplectic formalism. Since the polysymplectic formalism already contains the essence of our work and it clearly shows the relation to holomorphic symplectic geometry, we have decided to stick with this formalism for this article. The multisymplectic generalization is ongoing work of the authors, see also \cref{rem:multi}.
\end{remark}

\begin{definition}[\cite{gunther1987polysymplectic}]\label{def:polysymplectic}
    An $\R^2$-valued 2-form $\Omega\in\bigwedge^2T^*W\otimes\R^2$ on a manifold $W$ is called \emph{polysymplectic} if it is closed and non-degenerate in the following sense:
    \[
    \Omega(X,\cdot)=0 \Rightarrow X=0\text{ for all }X\in TW.
    \]
\end{definition}
Note that $\Omega$ is given by a pair of closed 2-forms $\eta_1,\eta_2$ on $W$. The non-degeneracy condition is equivalent to the requirement that $\ker\eta_1^\flat \cap \ker\eta_2^\flat=\{0\}$.
For example, on $\R^3$ with coordinates $(q,p_1,p_2)$ we may define the $\R^2$-valued form
\[
\Omega_{DW} = \eta_1\otimes\del_1+\eta_2\otimes\del_2,
\]
where $\R^2$ is spanned by the basis vectors $\del_1$ and $\del_2$ and
\begin{align*}
    \eta_1 = dp_1\wedge dq && \eta_2 = dp_2\wedge dq.
\end{align*}
Note that neither $\eta_1$ nor $\eta_2$ is non-degenerate, but their kernels have trivial intersection.

Whereas on symplectic manifolds there is a unique notion of a Hamiltonian, on a polysymplectic manifold there are actually two types of Hamiltonians, both of whom serve a different purpose. On the one hand, in symplectic geometry, a Hamiltonian gives rise to a vector field whose flow describes a differential equation. Since a polysymplectic form is $\R^2$-valued, its pairing with a vector field gives an element of $\R^2$. Thus, only for functions $W\to \R^2$ there is a chance of relating them to the flow of a vector field. In the polysymplectic world, we call these functions currents. On the other hand, the polysymplectic language is actually designed in order to deal with PDEs, whereas a flow of a vector field gives an ODE. In order to describe PDEs, we need to study Hamiltonians $W\to \R$. The following definition describes both of these types of functions.
\begin{definition}\label{def:currentsandHamiltonians}
    Let $(W,\Omega)$ be a polysymplectic manifold. 
    \begin{enumerate}[label=(\roman*)]
        \item A function $F:W\to\R^2$ is called a \emph{current}, if there exists a vector field $X_F$ on $W$ such that $\Omega(X_F,\cdot)=dF$.
        \item A function $H:W\to \R$ is called a \emph{Hamiltonian}.
    \end{enumerate}
\end{definition}

As mentioned, any current $F$ induces a flow along the trajectories of its associated vector field $X_F$. This flow preserves the polysymplectic form. In order to relate a Hamiltonian $H$ to a PDE, note that for any vector $v\in TW$ and any map $Z:\R^2\to W$, the map $\Omega(v,dZ\cdot)$ is a linear map from $\R^2$ to itself. Therefore, we can take its trace and denote it by $\Omega^\#(dZ)(v)$. Note that when $\Omega=\eta_1\otimes\del_1+\eta_2\otimes\del_2$ it follows that $\Omega^\#(dZ) = \eta_1(\cdot, \del_1Z)+\eta_2(\cdot,\del_2Z)$. The map $Z:\R^2\to W$ is called a solution of the polysymplectic Hamiltonian equation if 
\begin{align}\label{eq:polyDDW}
    dH = \Omega^\#(dZ) = \eta_1(\cdot,\del_1Z) + \eta_2(\cdot,\del_2Z).
\end{align}
In the example of $W=\R^3$ and $\Omega=\Omega_{DW}$ \cref{eq:DDW} is equivalent to \cref{eq:polyDDW}.

For Hamiltonians $H$ and currents $F$, one can define their Poisson bracket as $\{H,F\}:=-dH(X_F)$. The flow of $X_F$ preserves $H$ if and only if $\{H,F\}=0$. In other words, currents describe symmetries of Hamiltonians. We will see in \cref{sec:holomsympl} that in the complex-regularized polysymplectic framework there is a remarkable connection between currents and holomorphic Hamiltonian systems. 


\section{An alternative Hamiltonian formulation using Dirac operators}\label{sec:complexBridges}
Assume $M$ is either $\R^2$ or $\T^2$, so that $\{\del_1,\del_2\}$ forms a global frame of $TM$. As pointed out in \cref{sec:DDW}, we are looking for a Hamiltonian formulation of the nonlinear Laplace \cref{eq:Laplace}, that preserves its elliptic structure. Unfortunately, this is not possible for the Laplace equation on all spaces. However, when our space comes equipped with a complex structure, we may use it to define a regularized version of the polysymplectic from. 

To this extent, we now assume that $q:M\to\R^{2n}$ maps to even-dimensional Euclidean space equipped with its standard complex structure $i=\begin{pmatrix} 0 & -1\\ 1 & 0\end{pmatrix}$. We will write $q=(q_1,q_2):M\to\R^n\times\R^n$. The Laplace equation for $q$ can now be formulated as 
\begin{align}\label{eq:complexLaplace}
    -\Delta q = -4\del_{\bar{t}}\del_t q = \nabla V(q),
\end{align}
where $\del_t=\frac{1}{2}(\del_1-i\del_2)$ and $\del_{\bar{t}}=\frac{1}{2}(\del_1+i\del_2)$ are the holomorphic and anti-holomorphic derivatives and $V:\R^{2n}\to\R$.

To get to the De Donder-Weyl equations, we would introduce all four of the momentum vectors $p^\alpha_\beta=\del_\alpha q_\beta$, corresponding to the four different partial derivatives. However, we see from \cref{eq:complexLaplace} that the only combinations of partial derivatives we need, are given by the holomorphic derivative. Thus we may reduce the number of momentum variables to the number of coordinates $q$. Indeed, we define $p=2\del_t q$, meaning
\begin{align}\label{eq:pvariables}
\begin{split}
    p_1 &= \del_1q_1 +\del_2q_2\\
    p_2 &= \del_1q_2 - \del_2q_1,
    \end{split}
\end{align}
and get \cref{eq:complexLaplace} back by requiring $-2\del_{\bar{t}} p =\nabla V(q)$. That is
\begin{align}\label{eq:qvariables}\begin{split}
    \frac{\del V}{\del q_1} &= -\del_1p_1+\del_2p_2\\
    \frac{\del V}{\del q_2} &= -\del_1p_2 -\del_2p_1.
    \end{split}
\end{align}
Combined, we see that \cref{eq:complexLaplace} is equivalent to
\begin{align}\label{eq:complexBridges}
    \delslash Z := \begin{pmatrix} 0&-2\del_{\bar{t}}\\2\del_t&0\end{pmatrix}Z=\nabla H(Z),
\end{align}
where $Z=(q,p):M\to \R^{4n}$ and $H:\R^{4n}\to\R$ defined by $H(q,p)=\frac{1}{2}|p|^2+V(q)$.

In contrast to \cref{eq:DDW}, \cref{eq:complexBridges} is actually elliptic. One way to see that is by noting that $\delslash$ is a Dirac operator (clearly $\delslash^2=-\Delta$). Therefore, this Hamiltonian formulation is much more likely to be suitable for the generalization of elliptic methods from symplectic geometry. Indeed, in \cref{part:rigidity} of this paper, rigidity results will be proven based on these equations.

\begin{remark}
    We should note that equations (\ref{eq:pvariables}) and (\ref{eq:qvariables}) were already introduced by Bridges and Derks in \cite{bridgesderks}. In \cite{bridgesTEA}, Bridges views the variable that we call $q_2$ as a 'higher-momentum' variable. As was noted also by the authors of the present article in \cite{hyperkahlerrookworst}, there is a symmetry between the even and odd higher momenta, which leads to the viewpoint we chose for this article.
\end{remark}

Just like the De Donder-Weyl equations, \cref{eq:complexBridges} can be formulated in terms of a polysymplectic form. Define $\omega_1=dp_1\wedge dq_1+dp_2\wedge dq_2$ and $\omega_2=dp_1\wedge dq_2-dp_2\wedge dq_1$ and let 
\[
\Omega = \omega_1\otimes\del_1 + \omega_2\otimes\del_2.
\]
By a straightforward check, $\Omega$ defines a polysymplectic form on $\R^{4n}$ and \cref{eq:complexBridges} is equivalent to the polysymplectic equation
\begin{align*}
    dH=\Omega^\#(dZ) &=\omega_1(\cdot, \del_1Z)+\omega_2(\cdot,\del_2 Z).
\end{align*}
Note that in this case, not only $\Omega$ is non-degenerate, but so are its components $\omega_1$ and $\omega_2$. In fact, these components both define symplectic forms that are related by $\omega_2=-\omega_1(\cdot,I\cdot)$, where 
\begin{align*}
    I=\begin{pmatrix}
        i&0\\0&i^*
    \end{pmatrix}
    =\begin{pmatrix}
        i&0\\0&-i
    \end{pmatrix}
\end{align*}
is the standard complex structure on $\R^{4n}=\R^{2n}\times(\R^{2n})^*$.

\begin{remark}\label{rem:regpolvsholsym}
As a preparation for the next section in which we will explain the relation with holomorphic symplectic geometry, we want to note that $\Omega$ can be written also in a different way. Let $\omega^\C=d(p_1-ip_2)\wedge d(q_1+iq_2)$ be the holomorphic symplectic form\footnote{The minus sign for $p_2$ comes from the fact that $i^*=-i$.} on $\R^{2n}\times(\R^{2n})^*$ and $\bar{\omega}^\C$ its complex conjugate. Then we can rewrite 
\[\Omega=\omega^\C\otimes \del_t + \bar{\omega}^\C\otimes\del_{\bar{t}}.\]
This correspondence between the polysymplectic and holomorphic symplectic viewpoints holds in more generality and will be spelled out in \cref{prop:regpolvsholsym}. 
\end{remark}

The only geometric properties of $\R^{4n}$ that are essential for the definition of the regularized version of the De Donder-Weyl equations are its complex structure and the polysymplectic form $\Omega$. Thus we can extend to manifolds sharing these two structures.
\begin{definition}
    A \emph{complex-regularized polysymplectic structure} on a complex manifold $(W,I)$ is an $\R^2$-valued polysymplectic form $\Omega=\omega_1\otimes\del_1+\omega_2\otimes \del_2$ such that $\omega_2=-\omega_1(\cdot,I\cdot)$. The triple $(W,I,\Omega)$ is called a \emph{complex-regularized polysymplectic manifold}.
\end{definition}

\begin{example}\label{ex:cotangent}
    Our main example of a complex-regularized polysymplectic manifold is the cotangent bundle of a complex manifold. Let $(Q,i)$ be a complex manifold and equip $W=T^*Q$ with the induced complex structure $I$.
    Note that the projection $\pi:W\to Q$ is holomorphic. Then for $Z=(q,p)\in W$ and $X\in T_Z W$, we may define
    \begin{align*}
        (\lambda_1)_Z(X)&=p\circ d\pi(X)\\
        (\lambda_2)_Z(X)&=p\circ d\pi(IX).
    \end{align*}
    Then $\Omega = d\lambda_1\otimes\del_1-d\lambda_2\otimes\del_2$ is a complex-regularized polysymplectic structure on $(W,I)$, as can be checked in local coordinates. 
\end{example}

For the complex-regularized polysymplectic structure on $\R^{4n}$ defined above, we saw that both components are actually symplectic. We would like to prove a similar result in this more general setting.
\begin{proposition}
    Let $(W,I,\Omega=\omega_1\otimes\del_1+\omega_2\otimes\del_2)$ be a complex-regularized polysymplectic manifold, then both $\omega_1$ and $\omega_2$ are symplectic.
\end{proposition}
\begin{proof}
    By definition $\ker\omega_1^\flat\cap\ker\omega_2^\flat=0$ as $\Omega$ is a polysymplectic form. Suppose $X\in\ker\omega_1^\flat$. Then
    \begin{align*}
        \omega_2(X,Y) = -\omega_1(X,IY)=0,
    \end{align*}
    for all $Y\in TW$. So $X\in\ker\omega_1^\flat\cap\ker\omega_2^\flat=0$, yielding $X=0$. This proves $\omega_1$ is symplectic. A similar proof shows $\omega_2$ is symplectic as well.
\end{proof}

\begin{corollary}\label{cor:holomorphic}
    Let $\psi:(W,I,\Omega)\to(W',I',\Omega')$ be a diffeomorphism between complex-regularized polysymplectic manifolds, such that $\psi^*\Omega'=\Omega$. Then $\psi$ is holomorphic with respect to $I$ and $I'$.
\end{corollary}
\begin{proof}
    A simple calculation shows
    \begin{align*}
        \omega_1'(d\psi\cdot,I'\circ d\psi\cdot)&= -\omega_2'(d\psi\cdot,d\psi\cdot)\\
        &=-\omega_2\\
        &=\omega_1(\cdot,I\cdot)\\
        &=\omega_1'(d\psi\cdot,d\psi\circ I\cdot).
    \end{align*}
    As $d\psi$ is bijective and $\omega_1'$ is non-degenerate we get $I'\circ d\psi=d\psi\circ I$.
\end{proof}

\section{Holomorphic Hamiltonian systems and a Darboux theorem}\label{sec:holomsympl}
Recall from \cref{rem:regpolvsholsym}, that the polysymplectic form on $\R^{4n}$ defined in \cref{sec:complexBridges} can be written in terms of a holomorphic symplectic form. In this section we will elaborate more on the relation between complex-regularized polysymplectic and holomorphic symplectic geometry. For an introduction to holomorphic symplectic geometry, see \cite{wagner2023pseudo}. 

Recall that a holomorphic symplectic form on a complex manifold $(W,I)$ is a holomorphic 2-form on $W$ that is closed and whose restriction to the holomorphic tangent bundle $T^{(1,0)}W$ is non-degenerate. A holomorphic Hamiltonian system is a triple $(W,\omega^\C,F)$ where $F:W\to\C$ is holomorphic. In this case $F$ induces a holomorphic vector field $\X_F$ by $\omega^\C(\X_F,\cdot)=dF$. Our first result is that regularized polysymplectic structures are in one-to-one correspondence with holomorphic symplectic structures.
\begin{proposition}\label{prop:regpolvsholsym}
    For a complex manifold $(W,I)$ define the following two sets
    \begin{align*}
        \text{RegPol}(W,I)&:=\{\Omega\mid\Omega\text{ is a complex-regularized polysymplectic form}\}\\
        \text{HolSymp}(W,I)&:=\{\omega^\C\mid \omega^\C\text{ is a holomorphic symplectic form}\}.
    \end{align*}
    There is a bijection $\alpha:\text{RegPol}(W,I)\to\text{HolSymp}(W,I)$ given by
    \[\alpha(\omega_1\otimes\del_1+\omega_2\otimes\del_2)= \omega_1+i\omega_2\]
    with inverse given by
    \[\omega^\C\mapsto \omega^\C\otimes\del_t+\bar{\omega}^\C\otimes\del_{\bar{t}},\]
    where as always $\del_t=\frac{1}{2}(\del_1-i\del_2)$ and $\del_{\bar{t}}$ is its conjugate.
\end{proposition}
\begin{proof}
    Let $\Omega$ be a complex-regularized polysymplectic structure on $(W,I)$ and define $\omega^\C=\alpha(\Omega)$. Let $X^\pm\in T_\C W$ such that $IX^\pm=\pm i X^\pm$. Then
    \begin{align*}
        \omega^\C(\cdot, X^\pm)&= \omega_1(\cdot, X^\pm)+i\omega_2(\cdot,X^\pm)\\
        &=\omega_1(\cdot,X^\pm)-i\omega_1(\cdot, IX^\pm)\\
        &=\omega_1(\cdot,X^\pm)\pm\omega_1(\cdot,X^\pm).
    \end{align*}
    Indeed we see that $\iota_X\omega^\C=0$ for $X\in T^{(0,1)}W$ and $\omega^\C$ is non-degenerate on $T^{(1,0)}W$. At this point, $\omega^\C$ could still be any smooth (not necessarily holomorphic) section of $\Lambda^{2,0}W$. However, note that since $\omega^\C$ is closed it must be holomorphic for the following reason. The closedness means that $0=d\omega^\C=\del\omega^\C+\bar{\del}\omega^\C$. Since $\del\omega^\C$ is a section of $\Lambda^{3,0}W$, whereas $\bar{\del}\omega^\C$ is a section of $\Lambda^{2,1}W$ they both must vanish. We conclude that $\omega^\C$ is a holomorphic symplectic form, yielding that $\alpha$ is well-defined.\newline
    Vice versa, if $\omega^\C$ is a holomorphic symplectic form then its real and imaginary parts are real symplectic forms and satisfy $\text{Im}(\omega^\C)=-\text{Re}(\omega^\C)(\cdot,I\cdot)$, since $\omega^\C(\cdot,I\cdot)=i\omega^\C$ (see \cite{wagner2023pseudo}). So the map $\omega^\C\mapsto \omega^\C\otimes\del_t+\bar{\omega}^\C\otimes\del_{\bar{t}}$ is well-defined and is clearly the inverse of $\alpha$.
\end{proof}
From this proposition it might seem like complex-regularized polysymplectic geometry is equivalent to holomorphic symplectic geometry. However, even though the manifolds that both frameworks allow are the same, the equations that they study are different. While in polysymplectic geometry Hamiltonians are real-valued and yield PDEs, in holomorphic symplectic geometry one studies holomorphic Hamiltonians, giving rise to holomorphic vector fields. As it turns out, the holomorphic Hamiltonians from holomorphic symplectic geometry correspond precisely to the currents in complex-regularized polysymplectic geometry. 

Recall from \cref{def:currentsandHamiltonians} that a function $F:W\to\R^2$ on a polysymplectic manifold $(W,\Omega)$ is called a current if $dF_w$ lies in the image $\Omega^\flat_w(T_wW)$ for every $w\in W$. In this case we define $X_F$ by $X_F(w)=(\Omega^\flat_w)^{-1}(dF_w)$. Heuristically, currents describe symmetries of polysymplectic systems. On the other hand, holomorphic Hamiltonians describe symmetries of holomorphic symplectic systems. Given the correspondence in \cref{prop:regpolvsholsym} between complex-regularized polysymplectic and holomorphic symplectic forms, it should come as no surprise that currents in our polysymplectic setting correspond to holomorphic Hamiltonians. This is the content of the next lemma, which will be proven at the end of this section.
\begin{lemma}\label{lem:currents}
    Let $(W,I,\Omega)$ be a complex-regularized polysymplectic manifold and $\omega^\C=\alpha(\Omega)$. When we identify $\R^2\cong\C$ and $TW\cong T^{(1,0)}W$, we get that $F:W\to\R^2\cong \C$ is a current if and only if it is a holomorphic function and in that case $X_F=\X_F$. 
\end{lemma}

As a nice corollary of \cref{prop:regpolvsholsym}, we get a Darboux theorem for complex-regularized polysymplectic manifolds.
\begin{corollary}[Darboux theorem]\label{cor:Darboux}
    Let $(W,I,\Omega)$ be a complex-regularized polysymplectic manifold. Around every point $w\in W$ there are coordinates $q^i_1,q^i_2,p^i_1,p^i_2$ such that
    \begin{align*}
        I\frac{\del}{\del q^i_1}=\frac{\del}{\del q^i_2} && I\frac{\del}{\del p^i_1}=-\frac{\del}{\del p^i_2}
    \end{align*}                                                                                       
    and
    \begin{align*}
        \omega_1=\sum_i (dp^i_1\wedge dq^i_1+dp^i_2\wedge dq^i_2) && \omega_2=\sum_i (dp^i_1\wedge dq^i_2-dp^i_2\wedge dq^i_1).
    \end{align*}
\end{corollary}
\begin{proof}
    This follows immediately from \cref{prop:regpolvsholsym} and the Darboux theorem in \cite{wagner2023pseudo}.
\end{proof}

\begin{proof}[Proof of \cref{lem:currents}]
    As these statements are local in nature, we may use the Darboux theorem above to compute in local coordinates. Note the minus sign in the complex structure on the $p$-coordinates. An easy computation shows that the image of $\Omega^\flat_w$ is spanned by 
    \[\{dq^j_1\otimes\del_1+dq^j_2\otimes\del_2, dq^j_2\otimes\del_1-dq^j_1\otimes\del_2, dp^j_1\otimes\del_1-dp^j_2\otimes\del_2, dp^j_2\otimes\del_1+dp^j_1\otimes\del_2\}.\]
    That means that $dF$ lies in this image precisely when 
    \begin{align*}
        \frac{\del F_1}{\del q^j_1 } &= \frac{\del F_2}{\del q^j_2}\\
        -\frac{\del F_1}{\del q^j_2 } &= \frac{\del F_2}{\del q^j_1}\\
        -\frac{\del F_1}{\del p^j_1 } &= \frac{\del F_2}{\del p^j_2}\\
        \frac{\del F_1}{\del p^j_2 } &= \frac{\del F_2}{\del p^j_1},
    \end{align*}
    which are precisely the Cauchy-Riemann equations. 

    We also see that for currents $F$ the corresponding vector field $X_F$ is given by
\[
X_F = \begin{pmatrix}
    -\del_{p^j_1} F_1 \\ -\del_{p^j_2} F_1 \\ \del_{q^j_1} F_1 \\ \del_{q^j_2} F_1
\end{pmatrix}
= \begin{pmatrix}
    \del_{p^j_2} F_2 \\ -\del_{p^j_1} F_2 \\ \del_{q^j_2} F_2 \\ -\del_{q^j_1} F_2
\end{pmatrix}.
\]
On the other hand, the holomorphic vector field $\X_F$ is given by $(\X_F)_{q^j}=-\del_{p^j} F = -\del_{p^j_1}F_1-i\del_{p^j_2}F_1$ and $(\X_F)_{p^j}=\del_{q^j}F=\del_{q^j_1}F_1-i\del_{q^j_2}F_1$. The identification between the real and holomorphic tangent bundle of $W$ takes place by taking $\X_{q^j}=X_{q^j_1}+iX_{q^j_2}$ and $\X_{p^j}=X_{p^j_1}-iX_{p^j_2}$. The formulas above easily show that indeed $X_F$ and $\X_F$ coincide.
\end{proof}

\section{The Lagrange formalism}\label{sec:Lagrange}
The Hamiltonian formalism from \cref{sec:complexBridges} has a Lagrangian counterpart. Let $(Q,i)$ be a complex manifold and $L:\T^2\times TQ\to \R$ a Lagrangian on its tangent bundle depending on "time" $t\in\T^2$. To a map $q:\T^2\to Q$ we can assign its action\footnote{The factor of 2 is just put here in order to avoid dealing with extra factors of $\frac{1}{2}$ throughout.}
\begin{align*}
    \L(q):=\int_{\T^2} L(t,q,2\del_t q)\dV,
\end{align*}
where as before $\del_t=\frac{1}{2}(\del_1-i\del_2)$ and $\dV$ is the volume form on the torus. 

When $L$ is convex in the fibers, we can make the Legendre transformation 
\begin{align*}
    H(t,q,p) &= \max_{v\in T_qQ}\left(\langle p,v\rangle-L(t,q,v)\right) 
\end{align*}
By standard arguments the maximum is attained precisely when $p=\frac{\del L}{\del v}$,
where $\frac{\del L}{\del v}$ denotes the derivative of $L$ in the direction of the fiber. This defines a function $H:\T^2\times T^*Q\to \R$. 

\begin{lemma}
    In local coordinates, the Euler-Lagrange equations for $\L$ correspond to the Hamiltonian equations (\ref{eq:complexBridges}) for $H$. 
\end{lemma}
\begin{proof}
    We may assume $Q=\R^{2n}$ with its standard complex structure $i$. An element of $Q$ will be denoted by $(q_1,q_2)$ with $q_1,q_2\in \R^n$. Then the action is given by 
    \[
    \L(q) = \int_{\T^2}L(t,q_1,q_2,\del_1q_1+\del_2q_2,\del_1q_2-\del_2q_1)\dV.
    \]
    Let $q^s$ denote a family of maps $\T^2\to Q$, such that $q^0=q$ and $\frac{d}{ds}q^s|_{s=0}=\dot{q}$. Then 
    \begin{align*}
        d\L_q(\dot{q}) = \int_{\T^2} \left(\dot{q_1}\left(\frac{\del L}{\del q_1}-\del_1\frac{\del L}{\del v_1}+\del_2\frac{\del L}{\del v_2}\right)+\dot{q_2}\left(\frac{\del L}{\del q_2}-\del_1\frac{\del L}{\del v_2} -\del_2 \frac{\del L}{\del v_1}\right)\right)\dV,
    \end{align*}
    where it is implied that $L$ and its derivatives are evaluated at $(t,q,2\del_t q)$.
    Thus the Euler-Lagrange equations become
    \begin{align}\label{eq:EulerLagrange}\begin{split}
        \frac{\del L}{\del q_1}(t,q,2\del_t q) &= \del_1\frac{\del L}{\del v_1}(t,q,2\del_t q) - \del_2\frac{\del L}{\del v_2}(t,q,2\del_t q)\\
        \frac{\del L}{\del q_2}(t,q,2\del_t q) &= \del_1\frac{\del L}{\del v_2}(t,q,2\del_t q) + \del_2\frac{\del L}{\del v_1}(t,q,2\del_t q).
    \end{split}\end{align}
    Noting that 
    \begin{align*}
        \frac{\del H}{\del q_1} = -\frac{\del L}{\del q_1} && \frac{\del H}{\del p_1} = \del_1q_1+\del_2q_2\\
        \frac{\del H}{\del q_2} = -\frac{\del L}{\del q_2} && \frac{\del H}{\del p_2} = \del_1q_2-\del_2q_1
    \end{align*}
    and using that $p_i=\frac{\del L}{\del v_i}(t,q,2\del_t q)$ we get the equivalence with \cref{eq:complexBridges}.
\end{proof}

\begin{example}
    Our main example is a quadratic Lagrangian $L:\T^2\times TQ\to \R$ given by $L(t,q,v)=\frac{1}{2}|v|_g^2-V(q)$, where $g$ is a Riemannian metric on $Q$. It follows that $p=2\del_t q$ and $H(t,q,p)=\frac{1}{2}|p|_g^2+V(q)$. The Euler-Lagrange equations for this Lagrangian describe harmonic maps from $\T^2$ into $Q$. As we can see, this is indeed a natural generalization of the Lagrangian formalism for geodesics to the $d=2$ case.
\end{example}

Without referring to coordinates, the equivalence between the Hamilton and Lagrange formalisms can also be seen by a variational principle. 

\begin{proposition}\label{prop:var}
    Suppose $L:\T^2\times TQ\to \R$ is fibrewise convex and $H:\T^2\times T^*Q\to \R$ is its Legendre transform as above. Then a map $q:\T^2\to Q$ is an extremum of $\L$ if and only if $Z=(q,p):\T^2\to T^*Q$, given by $p=\frac{\del L}{\del v}(t,q,2\del_t q)$ is an extremum of the action functional
    \begin{align*}
        \A_H(Z)=\int_{\T^2}\left( \langle p,2\del_tq\rangle-H(t,q,p)\right)\, dt_1\wedge dt_2.
    \end{align*}
\end{proposition}
\begin{proof}
    The proof is analogous to the corresponding statement for $d=1$. The idea is that when $Z$ is extremal for $\A_H$, fiberwise $p$ must be an extremum of the function 
    \[p\mapsto \langle p,2\del_tq\rangle-H(t,q,p).\]
    Since the Legendre transform is a projection operator on the space of convex functions, this happens precisely when $p=\frac{\del L}{\del v}(t,q,2\del_t q)$ and $\langle p,2\del_tq\rangle-H(t,q,p)=L(t,q,p)$. 
\end{proof}

Note that the term $\langle p,2\del_t q\rangle$ in the action functional may also be defined intrinsically by using the 1-forms $\lambda_1$ and $\lambda_2$ on $T^*Q$ that were defined in \cref{ex:cotangent}. Indeed
\begin{align*}
    \langle p,2\del_t q\rangle &= p\circ d\pi(\del_1 Z-I\del_2 Z)=\lambda_1(\del_1Z)-\lambda_2(\del_2 Z).
\end{align*}
In \cref{sec:actionfunctional}, we will see a more general definition of the action functional on polysymplectic manifolds.

\section{The best of both worlds}\label{sec:applications}
The motivation given at the start of this article for introducing the complex-regularized polysymplectic language, was to take advantage of "ellipticity" of holomorphic symplectic geometry, whilst at the same time keep the connection with PDEs and harmonic maps from polysymplectic geometry. In this section we want to give examples of two problems in which we can clearly see the use for connecting these two worlds. This section is not meant to give rigorous proofs, but rather aims to describe the ideas involved.

\subsubsection*{Obstructions to holomorphic Lagrangian embeddings}
Given a holomorphic symplectic manifold $W$ of dimension $4n$ and a closed complex manifold $L$ of dimension $2n$, a natural question to ask is whether there exists a holomorphic Lagrangian embedding $L\hookrightarrow W$, i.e. a holomorphic embedding such that the pullback of the holomorphic symplectic form on $W$ vanishes on $L$. Without going into details, we will argue that this question is related to Morse theory of harmonic maps and therefore naturally connects to the polysymplectic world.

The analogous question for $d=1$ is whether an $n$-dimensional manifold $\L$ has a Lagrangian embedding into a $2n$-dimensional symplectic manifold $\W$. It is well-known that obstructions to the existence of such a Lagrangian embedding stem from the Morse theory of geodesics on $\L$ (see for example \cite{latschev2014fukaya}). When a Lagrangian embedding $\L\hookrightarrow\W$ exists one can define the evaluation map from the moduli space of holomorphic disks with boundary in $\L$ to the loop space $\Lambda\L=\{\ell:S^1\to \L\}$ of $\L$
\[
\text{ev}:\M=\{u:(D^2,\del D^2=S^1)\to (\W,\L)\mid u\text{ is holomorphic}\}\to\Lambda\L.
\]
This gives a homology chain $\text{ev}_*[\M]\in C_*(\Lambda\L)$. The homology of $\Lambda\L$ can in turn be described using the Morse theory of closed geodesics in $\L$. The situation is depicted in \cref{fig:realLagrangian}. The dotted lines denote a Morse gradient flow line from a closed geodesic $\gamma$ to the boundary of a holomorphic disk $(D^2,\del D^2)\to(\W,\L)$.

\begin{figure}[h]
\centering
\begin{tikzpicture}[x=0.75pt,y=0.75pt,yscale=-1,xscale=1]
\draw    (100,66) -- (48,227) ;
\draw    (48,227) -- (348,227) ;
\draw   (126,115) .. controls (126,103.95) and (145.25,95) .. (169,95) .. controls (192.75,95) and (212,103.95) .. (212,115) .. controls (212,126.05) and (192.75,135) .. (169,135) .. controls (145.25,135) and (126,126.05) .. (126,115) -- cycle ;
\draw   (185,181) .. controls (185,169.95) and (204.25,161) .. (228,161) .. controls (251.75,161) and (271,169.95) .. (271,181) .. controls (271,192.05) and (251.75,201) .. (228,201) .. controls (204.25,201) and (185,192.05) .. (185,181) -- cycle ;
\draw  [dash pattern={on 4.5pt off 4.5pt}]  (128,121) -- (187,188) ;
\draw  [dash pattern={on 4.5pt off 4.5pt}]  (210,107) -- (269,174) ;
\draw    (129,107) .. controls (155,55) and (183,53) .. (210,107) ;

\draw (196,65.4) node [anchor=north west][inner sep=0.75pt]    {$D^{2}$};
\draw (165,113.4) node [anchor=north west][inner sep=0.75pt]    {$S^{1}$};
\draw (272,177.4) node [anchor=north west][inner sep=0.75pt]    {$\gamma $};
\draw (60,202.4) node [anchor=north west][inner sep=0.75pt]    {$\mathcal{L}$};
\draw (37,108.4) node [anchor=north west][inner sep=0.75pt]    {$\mathcal{W}$};
\end{tikzpicture}
\caption{\label{fig:realLagrangian}Lagrangian embedding $\mathcal{L}\subseteq\mathcal{W}$ with holomorphic disk $(D^2,S^1)\to (\mathcal{W},\mathcal{L})$ and geodesic $\gamma$.}
\end{figure}
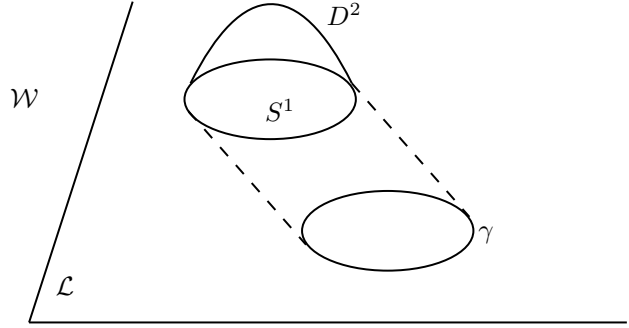

Going back to the situation for $d=2$, it is natural to consider a similar construction. The role of holomorphic disks is now taken over by Fueter disks $f:(D^3,\del D^3=S^2)\to(W,L)$, where $W$ is the holomorphic symplectic manifold and $L$ is an embedded holomorphic Lagrangian. This means that the evaluation map now takes values in the space of maps $S^2\to L$. The homology of this space can be described using the Morse theory of harmonic spheres. The situation is depicted in \cref{fig:complexLagrangian}.

\begin{figure}[h]
\centering
\begin{tikzpicture}[x=0.75pt,y=0.75pt,yscale=-1,xscale=1]

\draw    (100,66) -- (48,227) ;
\draw    (48,227) -- (348,227) ;
\draw   (126,115) .. controls (126,103.95) and (145.25,95) .. (169,95) .. controls (192.75,95) and (212,103.95) .. (212,115) .. controls (212,126.05) and (192.75,135) .. (169,135) .. controls (145.25,135) and (126,126.05) .. (126,115) -- cycle ;
\draw   (185,181) .. controls (185,169.95) and (204.25,161) .. (228,161) .. controls (251.75,161) and (271,169.95) .. (271,181) .. controls (271,192.05) and (251.75,201) .. (228,201) .. controls (204.25,201) and (185,192.05) .. (185,181) -- cycle ;
\draw  [dash pattern={on 4.5pt off 4.5pt}]  (128,121) -- (187,188) ;
\draw  [dash pattern={on 4.5pt off 4.5pt}]  (210,107) -- (269,174) ;
\draw    (129,107) .. controls (155,55) and (183,53) .. (210,107) ;
\draw  [dash pattern={on 0.84pt off 2.51pt}] (176.45,99.92) .. controls (182.5,107.34) and (184.07,120.1) .. (179.95,128.42) .. controls (175.84,136.75) and (167.6,137.49) .. (161.55,130.08) .. controls (155.5,122.66) and (153.93,109.9) .. (158.05,101.58) .. controls (162.16,93.25) and (170.4,92.51) .. (176.45,99.92) -- cycle ;
\draw  [dash pattern={on 0.84pt off 2.51pt}] (235.45,165.92) .. controls (241.5,173.34) and (243.07,186.1) .. (238.95,194.42) .. controls (234.84,202.75) and (226.6,203.49) .. (220.55,196.08) .. controls (214.5,188.66) and (212.93,175.9) .. (217.05,167.58) .. controls (221.16,159.25) and (229.4,158.51) .. (235.45,165.92) -- cycle ;

\draw (60,202.4) node [anchor=north west][inner sep=0.75pt]    {$L$};
\draw (37,108.4) node [anchor=north west][inner sep=0.75pt]    {$W$};
\draw (196,56.4) node [anchor=north west][inner sep=0.75pt]    {$D^{3}$};
\draw (185,108.4) node [anchor=north west][inner sep=0.75pt]    {$S^{2}$};
\draw (278,174.4) node [anchor=north west][inner sep=0.75pt]    {$Z$};
\end{tikzpicture}
\caption{\label{fig:complexLagrangian}Complex Lagrangian embedding $L\subseteq W$ with Fueter disk $(D^3,S^2)\to (W,L)$ and harmonic sphere $Z$.}
\end{figure}
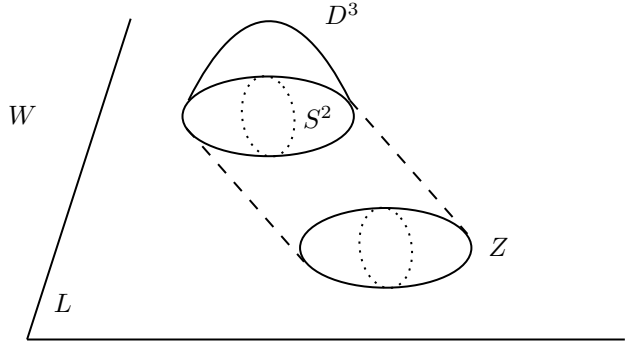

\begin{remark}\label{rem:multi}
    Note that for this construction a theory of harmonic spheres is required, whereas so far we only dealt with harmonic tori. As the polysymplectic language requires a base manifold with trivial tangent bundle, harmonic spheres cannot be described by it. They can however be defined in the multisymplectic framework, which studies the 3-form
    \[\wt{\Omega}=\omega_1\wedge dt_2-\omega_2\wedge dt_1\]
    coming from contracting $\Omega$ with the volume-form on the time-manifold (note that a primitive of this 3-form is used in the definition of the action functional in \cref{sec:actionfunctional}). The multisymplectic formalism can be defined for maps from any surface and even extends to sections of non-trivial bundles over a surface. See \cite{deLeonMethods} for the translation between the poly- and multisymplectic formalisms.
\end{remark}

\subsubsection*{Harmonic maps with boundary}
The second problem we want to highlight is the study of harmonic maps with prescribed boundary conditions. This is actually a problem that is natural to ask in the polysymplectic framework, but as we shall see holomorphic symplectic geometry will pop up naturally here. 

First, let's look again at the $d=1$ case. For a Riemannian manifold $(\mathcal{Q},g)$ one can study the existence of geodesics with prescribed boundary points $q_0,q_1\in\mathcal{Q}$. Geodesics arise as solution to the Hamiltonian equations on $T^*\mathcal{Q}$ for $H(q,p)=\frac{1}{2}|p|_g^2$. The boundary conditions amount to fixing the ends of the Hamiltonian path to the Lagrangian submanifolds $T^*_{q_0}\mathcal{Q}$ and $T^*_{q_1}\mathcal{Q}$ in $T^*\mathcal{Q}$. Existence results then follow from studying the Lagrangian Floer homology. More generally, one can look for geodesics whose boundary points are confined to some submanifolds $\mathcal{R}_1,\mathcal{R}_2\subseteq\mathcal{Q}$. On the Hamiltonian side, this imposes the condition that the Hamiltonian paths should start and end on the conormal bundles $N^*\mathcal{R}_1$ and $N^*\mathcal{R}_2$ respectively (see \cite{ASSELLE20163513}), which are also Lagrangian submanifolds. 

For $d=2$, we start with a K\"ahler manifold $(Q,i,g)$ and ask for the existence of a harmonic map $\Sigma\to Q$ with boundary $\del\Sigma$ on some submanifold $R\subseteq Q$. As seen before, the equation describing harmonic maps can be described by the complex-regularized polysymplectic Hamiltonian $H(q,p)=\frac{1}{2}|p|^2_g$ on $T^*Q$ with its standard complex structure $I$ and symplectic forms $\omega_1$ and $\omega_2$. Now, the boundary conditions translate into fixing the boundary of $Z:\Sigma\to T^*Q$ to the conormal bundle $N^*R$. Just as before, the conormal bundle is Lagrangian with respect to $\omega_1$, but depending on the choice of $R$, it may have more structure. For example, if $R\subseteq Q$ is a complex submanifold, then $N^*R\subseteq T^*Q$ is an $I$-complex submanifold that is Lagrangian with respect to both $\omega_1$ and $\omega_2$, or in other words, it is a holomorphic Lagrangian with respect to the holomorphic symplectic form. On the other hand, if $R\subseteq Q$ is a maximal totally real submanifold, then there exists a metric $\langle\cdot,\cdot\rangle$ on $T^*Q$, such that $N^*R$ is Lagrangian with respect to the symplectic forms $\langle\cdot,I\cdot\rangle$ and $\omega_1$. This implies also that $N^*R$ is $K$-complex, where $K$ is defined by $\omega_2=\langle\cdot,K\cdot\rangle$.


\part{Rigidity results}\label{part:rigidity}
\section{Polysymplectic non-squeezing}\label{sec:nonsqueezing}
The first rigidity result we want to present in this context is Gromov's non-squeezing. Note that a diffeomorphism $\psi:W\to W$ of a complex-regularized polysymplectic manifold preserves $\Omega$ if and only if it preserves both of the symplectic forms $\omega_i$. This yields the following non-squeezing result.
\begin{theorem}\label{thm:nonsqueezing}
    Let $\psi:\R^{4n}\to\R^{4n}$ be a diffeomorphism preserving the standard complex-regularized polysymplectic form on $\R^{4n}$. Also, let $B^{4n}_r\subseteq\R^{4n}$ be the ball of radius $r$ and $B^2_R\subseteq\R^2$ the ball of radius $R$ in the $(q^j_\alpha,p^j_\beta)$-plane for some $\alpha,\beta\in\{1,2\}$ and $j\in\{1,\ldots,n\}$. If 
    \[\psi(B^{4n}_r)\subseteq B^2_R\times\R^{4n-2},\]
    then $r\leq R$.
\end{theorem}
\begin{proof}
    Note that the $(q^j_1,p^j_1)$ and $(q^j_2,p^j_2)$-planes are symplectic with respect to $\omega_1$ and the $(q^j_1,p^j_2)$ and $(q^j_2,p^j_1)$-planes are symplectic with respect to $\omega_2$. The result follows immediately from Gromov's non-squeezing.
\end{proof}

Even though this non-squeezing result follows immediately from the known theory for symplectic forms, it is important to note that this result definitely does not hold for general polysymplectic forms. To the authors' knowledge, this is the first version of non-squeezing that is formulated for a class of polysymplectic manifolds. It is also worth mentioning that for $n=1$, this result coincides with the lowest dimensional case of a non-squeezing conjecture by Oh for holomorphic volume preserving maps (\cite{oh2002holomorphic}).

As a consequence, it is now possible to study $C^0$-limits of polysymplectomorphisms. 

\begin{corollary}\label{cor:C0limits}
    Let $(W,I,\Omega)$ be a complex-regularized polysymplectic manifold and let $\psi_\nu:W\to W$ be a sequence of diffeomorphisms preserving $\Omega$. Assume that the $\psi_\nu$ converge to a diffeomorphism $\psi:W\to W$ in the $C^0$-topology. Then $\psi^*\Omega=\Omega$.
\end{corollary}
\begin{proof}
    This follows immediately from the analogous statement for symplectomorphisms. For a proof see \cite{mcduff2017introduction}.
\end{proof}

\begin{remark}
    By \cref{cor:holomorphic} the limit diffeomorphism $\psi$ in \cref{cor:C0limits} is holomorphic. In general, it is already well-known that $C^0$-limits of holomorphic functions are holomorphic (see \cite{stein2003complex}).
\end{remark}

\section{The action functional and a cuplength result}\label{sec:actionfunctional}
Given a complex-regularized polysymplectic form there is a natural action functional on maps $\T^2\to W$. Assume that $\Omega=d\Theta$ is exact and let $d\V=dt_1\wedge dt_2=-\frac{1}{2i}dt\wedge d\bar{t}$ be the volume form on $\T^2$. We may contract $\Theta=\theta_1\otimes\del_1+\theta_2\otimes\del_2$ with $\dV$ to give
\[
\wt{\Theta}:=\theta_1\wedge dt_2-\theta_2\wedge dt_1\in\Lambda^2(W\times \T^2).
\]
For $Z:\T^2\to W$, define $\wt{Z}=Z\times\text{id}:\T^2\to W\times\T^2$ and
\begin{align*}
    \A(Z) &=\int_{\T^2}\wt{Z}^*\wt{\Theta}\\
    &=\int_{\T^2} (\theta_1(\del_1Z)+\theta_2(\del_2Z))d\V.
\end{align*}
Let $\A_H(Z)=\A(Z)-\int_{\T^2}H(Z)d\V$. The critical points of $\A_H$ are precisely the solutions to the polysymplectic equation $dH=\Omega^\sharp(dZ)$. Note that this action functional is real-valued, making it possible to study gradient lines, whereas in holomorphic symplectic geometry one needs to make choices in order to get a real-valued action. Also, when $W=T^*Q$ and $\Omega$ is the complex-regularized polysymplectic form defined in \cref{ex:cotangent}, this action functional coincides with the one given in \cref{prop:var}.

To study gradient lines of the action functional, we need to put a metric on the space of maps from $\T^2$ to $W$. Just as in Floer theory, we will put an $L^2$-metric on this space, coming from a metric on $W$. The next proposition asserts the existence of the type of metric that we want to use.

\begin{proposition}\label{prop:metricsmooth}
    For every complex-regularized polysymplectic manifold $(W,I,\omega_1\otimes\del_1+\omega_2\otimes\del_2)$ there exists a Riemannian metric $g$ and anti-commuting almost complex structures $J$ and $K$ on $W$, such that 
    \begin{align*}
        \omega_1 &= g(\cdot,J\cdot)\\
        \omega_2 &= g(\cdot,K\cdot)\\
        K &= IJ.
    \end{align*}
\end{proposition}
\begin{proof}
    See \cref{app:metric}.
\end{proof}

Pick a metric $g$ on $W$ as in \cref{prop:metricsmooth}. Then the $L^2$-gradient of $\A$ is given by
\[\grad\A(Z) = J\del_1Z+K\del_2Z=:\delslash Z,\]
meaning that its negative gradient lines are described by the Fueter equation
\begin{align*}
    I\del_sZ+K\del_1 Z-J\del_2Z=0.
\end{align*}
Thus we may expect the role of holomorphic curves in symplectic geometry to be replaced by Fueter maps in complex-regularized polysymplectic geometry. The Fueter equation is the quaternionic analogue of the Cauchy-Riemann equations. Just as for holomorphic curves, there are compactness and Fredholm theories developed for Fueter maps (see e.g. \cite{HNS,walpuski2017compactness}). 

Gradient lines of the action functional above may be used to prove a cuplength result for solutions of \cref{eq:complexBridges} in the same spirit as Arnold's conjecture. 
\begin{theorem}\label{thm:cuplength}
    Let $Q=\T^{2n}$ with its standard complex structure $i$ and $W=T^*Q$ with the standard complex-regularized polsymplectic form $\Omega$. Denote elements of $W$ by $Z=(q,p)$. If $H:\T^2\times W\to \R$ is given by $H(t,q,p)=\frac{1}{2}|p|^2+h(t,q,p)$, where $h:\T^2\times W\to\R$ is a smooth function with finite $C^2$-norm, then the complex-regularized polysymplectic equation
    \begin{align}\label{eq:Ham}dH=\Omega^\#(dZ)\end{align}
    has at least $(2n+1)$ solutions. In particular, when choosing $h(t,q,p)=V(q)$, for $V:Q\to\R$, it follows that the non-linear Laplace equation
    \[-\Delta q = \nabla V(q)\]
    has at least $(2n+1)$ solutions.
\end{theorem}

In the next sections, \cref{thm:cuplength} will be proven using a Floer theoretic argument along the lines of \cite{schwarz1998quantum,albers2016cuplength}. Note that similar results have been proven in \cite{ginzburg2012hyperkahler,ginzburg2013arnold} for closed target manifolds that arise as smooth compact quotients of a vector space. Also for three-dimensional time manifolds similar results have been proven in \cite{HNS} (for closed target manifolds) and \cite{hyperkahlerrookworst} (for cotangent bundles of tori). In the latter article, the proof was given as a corollary of results from \cite{ginzburg2012hyperkahler} after establishing the necessary $C^0$-bounds. This proof does not use Floer theoretic arguments. In the present article we aim to establish a connection to the Floer theory of these systems and will prove the necessary Fredholm and compactness results (see \cref{sec:Fredholm,sec:Compactness} respectively).

In \cref{thm:cuplength}, we restrict to cotangent bundles of tori for simplicity and because it already highlights much of the necessary analysis. The reason to prove this theorem using Floer theory though is that we expect this proof to generalize to cotangent bundles of more general K\"ahler manifolds. As noted in \cite{HNS}, the bubbling off analysis becomes more involved than it is in the $d=1$ case, since we now have to deal with compactness of Fueter curves. A different strategy for proving the existence of solutions to the non-linear Laplace equation, would be to apply Morse theory to the Lagrangian $\L(q)=\int_{\T^2}\left(\frac{1}{2}|dq|^{2}-V(q)\right)\dV$. However, note that $\L$ is defined on $H^1(\T^2,Q)$ which is not a Banach manifold for general $Q$. Therefore, Morse theory is not suitable to prove these existence result, whereas the proof using Floer curves is not affected by this. 

\begin{remark}
    Note that in \cref{thm:cuplength} it is essential to study real-valued Hamiltonians rather then holomorphic Hamiltonian (i.e. currents). To the authors' knowledge there is no version of the Arnold conjecture for holomorphic Hamiltonian systems. Observe in particular that the maximum principle prohibits the existence of bounded holomorphic non-linearities.
\end{remark}


\section{$C^0$-bounds and Floer maps}\label{sec:proofArnold}
First, we will prove the necessary $C^0$-bounds. The proofs are analogous to the ones given in \cite{hyperkahlerrookworst}, but as the setting is slightly different, the proofs will briefly be repeated below. We denote 
\begin{align*}
    \mathcal{P}(H)&=\{Z:\T^2\to T^*\T^{2n}\mid Z\textup{ is nullhomotopic and satisfies \cref{eq:Ham}}\}\\
    \mathcal{P}_a(H)&=\{Z\in\mathcal{P}(H)\mid a\geq \A_H(Z)\}.
\end{align*}

\begin{lemma}\label{lem:Linfty}
    Let $H_{t}(Z)=\frac{1}{2}|p|^2+h_t(Z)$, where $h$ is smooth, real-valued and has finite $C^2$-norm. Then for any $a\in\R$ the set $\mathcal{P}_a(H)$ is bounded in $L^\infty$. 
\end{lemma}
\begin{proof}
    We start by proving that $\mathcal{P}_a(H)$ is bounded in $L^2$. Recall that for solutions of \cref{eq:Ham} we get $2\del_t q=\frac{\del H}{\del p}=p+\frac{\del h}{\del p}$. Therefore, if $Z$ is a solution, we get
    \begin{align*}
        \A_H(Z) &= \int_{\T^2} \left( \langle p,p+\frac{\del h_t}{\del p}(Z)\rangle -H_{t}(Z)\right)\dV\\
        &=\int_{\T^2}\left( \frac{1}{2}|p|^2 +\langle p,\frac{\del h}{\del p}(Z)\rangle-h_t(Z)\right)\dV\\
        &\geq c_0||p||_{L^2}^2-c_1,
    \end{align*}
    for some $c_0>0$ and $c_1\geq 0$. The last inequality comes from the boundedness of $h$ and its derivatives. The assumption $a\geq\A_H(Z)$ now yields a uniform $L^2$-bound on $p$. As $q$ maps into $\T^{2n}$ it is $L^2$-bounded by construction. So indeed $\mathcal{P}_a(H)$ is bounded in $L^2$. 

    We now proceed by a bootstrapping argument. For $Z\in\mathcal{P}_a(H)$ we have\footnote{In the first line we use partial integration, using that $Z$ is nullhomotopic.}
    \begin{align*}
        ||dZ||^2_{L^2} &= \int_{\T^2}\left(|2\del_t q|^2+|2\del_{\bar{t}} p|^2\right)\dV\\
        &=\int_{\T^2}\left(\left|\frac{\del H}{\del p}\right|^2+\left|\frac{\del H}{\del q}\right|^2\right)\dV\\
        &\leq ||p||_{L^2}^2 + \int_{\T^2}\left(\left|\frac{\del h}{\del q}\right|^2+\left|\frac{\del h}{\del p}\right|^2\right)\dV.
    \end{align*}
    As the first derivatives of $h$ are bounded and $p$ is uniformly bounded in $L^2$ this induces a uniform $L^2$-bound on $dZ$. Therefore $\mathcal{P}_a(H)$ is bounded in $H^1$. A similar argument shows that the $H^1$-norm of $dZ$ is bounded by the $H^1$-norm of $Z$ and the $C^2$-norm of $h$. Thus, we get that $\mathcal{P}_a(H)$ is bounded in $H^2$. As $2\cdot 2>2$ we get boundedness in $L^\infty$ by Sobolev's embedding theorem. 
\end{proof}

From \cref{lem:Linfty} it follows that we can actually use a cut-off version of $h$. Let $\chi_\rho:[0,\infty)\to\R$ a smooth cut-off function such that $\chi_\rho\equiv 1$ on $[0,\rho-1]$ and $\chi_\rho\equiv 0$ on $[\rho,\infty)$. We may assume there is a bound on $||\chi_\rho||_{C^2}$ independent of $\rho$. Define the cut-off Hamiltonian as $\tilde{H}^\rho_{t}(Z):=\frac{1}{2}|p|^2+\tilde{h}^\rho_{t}(Z):=\frac{1}{2}|p|^2+\chi_\rho(|p|)h_{t}(Z)$, which is smooth for $\rho>1$
\begin{corollary}\label{cor:cutoffHam}
    For any $a\in\R$ there exists some $\rho>0$ such that $\mathcal{P}_a(H)=\mathcal{P}_a(\tilde{H}^\rho)$.
\end{corollary}
\begin{proof}
    Note that $\tilde{h}^\rho$ has uniformly bounded $C^2$-norm. The proof of \cref{lem:Linfty} shows that there are constants $C,\wt{C}>0$ such that for any $Z\in\mathcal{P}_a(H)$ and $\wt{Z}^\rho\in\mathcal{P}_a(\tilde{H}^\rho)$, we have
    \begin{align*}
        ||Z||_{L^\infty}<C && ||\wt{Z}^\rho||_{L^\infty}<\wt{C}
    \end{align*}
    for any $\rho>1$. Choosing $\rho>\max\{C,\wt{C}\}+1$, we get that $\mathcal{P}_a(H)=\mathcal{P}_a(\tilde{H}^\rho)$.    
\end{proof}

\subsection*{Floer maps}
Now that we have the necessary bounds on the orbits, we turn to Floer maps. A Floer map is defined as a negative $L^2$-gradient flow line of the action and is given by
\begin{align}\label{eq:Floer}
\del_sZ+J\del_1Z+K\del_2Z=\nabla \wt{H}^\rho(Z)
\end{align}
where we have taken the Hamiltonian to be $\tilde{H}^\rho$ in view of \cref{cor:cutoffHam}. Note that the standard metric on $T^*\T^{2n}$ is compatible with both $\omega_1$ and $\omega_2$ and yields
\begin{align*}
    J=\begin{pmatrix}
        0 & -\Id_{2n}\\ \Id_{2n} & 0
    \end{pmatrix} && K=\begin{pmatrix}
        0 & -i\\ -i & 0
    \end{pmatrix}.
\end{align*}
Therefore the Floer maps are given by
\begin{align*}
    \del_s q-2\del_{\bar{t}}p&=\frac{\del \wt{H}^\rho}{\del q}\\
    \del_sp +2\del_t q&=\frac{\del \wt{H}^\rho}{\del p}.
\end{align*}

The following lemma shows that the image of a Floer map actually sits within a compact subset of $T^*\T^{2n}$ (compare with Lemma 3.4 from \cite{hyperkahlerrookworst}).
\begin{lemma}\label{lem:Floerbounds}
    Let $Z:\R\times\T^2\to T^*\T^{2n}$ be a solution to \cref{eq:Floer} which converges to critical points of $\A_{\tilde{H}^\rho}$ for $|s|\to\infty$. Then $|p(s,t)|\leq \rho$ for any $(s,t)\in\R\times\T^2$.
\end{lemma}
\begin{proof}
    Suppose $|p(s_0,t_0)|>\rho$ for some $(s_0,t_0)\in\R\times \T^2$. Locally around $(s_0,t_0)$ the Hamiltonian reduces to $\tilde{H}^\rho(Z(s,t))=\frac{1}{2}|p(s,t)|^2$. Therefore, on this neighborhood
    \begin{align*}
        \del_s^2 p &=\del_s p - 2\del_s\del_{t}{q}\\
        &= \del_s p - 4\del_t\del_{\bar{t}} p.
    \end{align*}
    It follows that 
    \begin{align*}(\del_s^2 +\Delta)p = \del_s^2 p + 4\del_t\del_{\bar{t}}p=\del_s p.\end{align*}

    Define $\phi=\frac{1}{2}|p|^2$. Then
    \begin{align*}
        (\del_s^2+\Delta)\phi &\geq \del_s \phi.
    \end{align*}
    The maximum principle now tells us that $\phi$ cannot attain a maximum in the neighbourhood of $(s_0,t_0)$, contradicting the assumption that the Floer map converges as $|s|\to\infty$. 
\end{proof}

Having established the $C^0$-bounds we can prove \cref{thm:cuplength}. The necessary Fredholm and compactness results will be deferred to \cref{sec:Fredholm,sec:Compactness}. For the proof we will follow the lines of the paper \cite{albers2016cuplength}. 

\begin{proof}[Proof of \cref{thm:cuplength}]
To start the proof, assume there are finitely many solutions to \cref{eq:Ham}, since otherwise we would be done. Then their action is necessarily bounded. \Cref{cor:cutoffHam} implies that we may replace $H$ by $\tilde{H}^\rho$ for some $\rho$ and by \cref{lem:Floerbounds} also the Floer maps will take value in some compact set $\T^{2n}\times B\subseteq T^*\T^{2n}$. Let $M\subseteq C^\infty(\T^2,\T^{2n}\times B)$ be the subset of nullhomotopic maps. Just like in \cite[\pp 3.1]{albers2016cuplength} we define a smooth family of functions $\beta_r:\R\to[0,1]$ for $r\geq 0$ such that 
\begin{itemize}
    \item $\beta_r(s)=0$ for $s\leq -1$ and $s\geq (2n+1)r+1$ for all $r\geq 0$,
    \item $\beta_r|_{[0,(2n+1)r]}\equiv 1$ for $r\geq 1$,
    \item $0\leq \beta_r'(s)\leq 2$ for $s\in(-1,0)$ and \\$0\geq \beta_r'(s)\geq -2$ for $s\in((2n+1)r,(2n+1)r+1)$,
    \item $\lim_{r\to 0^+}\beta_r=0$ in the strong topology on $C^\infty$.
\end{itemize}
See \cite{albers2016cuplength} for the precise conditions on $\beta_r$ and \cref{fig:beta} for an idea of what they look like. We will use the functions $\beta_r$ to interpolate between the quadratic Hamiltonian with and without non-linearity. To this extent define
\[
\tilde{H}^{\rho,r}_{s,t}(Z)=\frac{1}{2}|p|^2+\beta_r(s)\tilde{h}^\rho_{t}(Z).
\]
One can see the Hamiltonian $\tilde{H}^{\rho,r}$ as the quadratic Hamiltonian $H^0(Z)=\frac{1}{2}|p|^2$ with the non-linearity $\tilde{h}^\rho$ switched on just on an interval whose length is determined by $r$. Correspondingly, we have the action $G_{r,s}:=\A_{\tilde{H}^{\rho,r}}$, interpolating between the actions $\A_{H^0}$ and $\A_{\tilde{H}^\rho}$. 

\begin{figure}[h]
    \centering
    \includegraphics[width=0.8\linewidth]{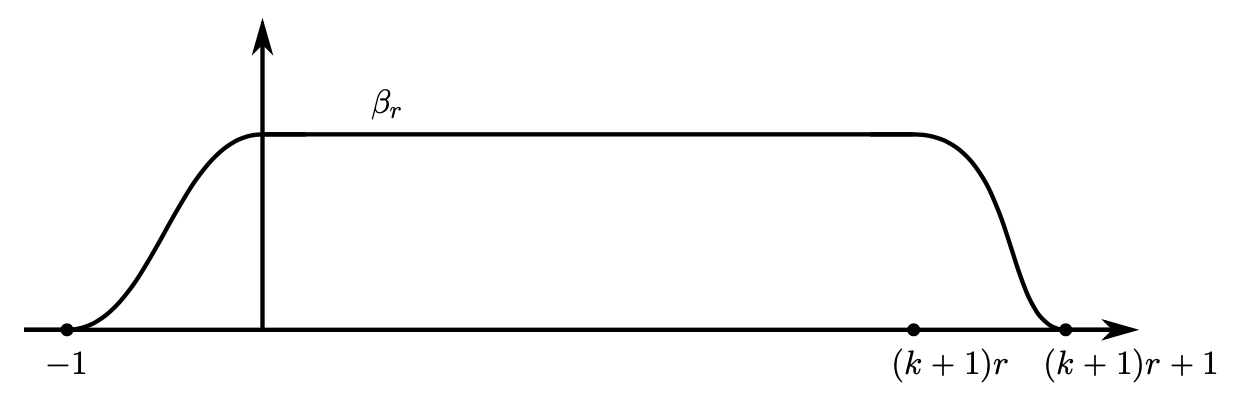}
    \caption{The functions $\beta_r$. Source: \cite{albers2016cuplength}. In our case we take $k=\textup{cl}(\T^{2n})=2n$.}
    \label{fig:beta}
\end{figure}

With this setup we can introduce the moduli space
\begin{align}\label{eq:ModuliSpace}
    \M = \{(r,Z)\mid r\geq 0, Z:\R\to M, (\star 1)-(\star3)\},
\end{align}
where $(\star1)-(\star3)$ are given below. Note that we can view $Z$ as a contractible map $\R\times\T^2\to\T^{2n}\times B$. The condition $(\star1)$ is that $Z$ is a negative gradient line of $G_{r,s}$, meaning a Floer map for $\tilde{H}^{\rho,r}$. Condition $(\star2)$ is finiteness of energy
\begin{align}
    \tag{$\star2$} E(Z) = \int_{-\infty}^\infty\int_{\T^2}|\del_s Z(s,t)|^2\dV\,ds <+\infty
\end{align}
and the final condition is convergence of the $p$-coordinates
\begin{align}
    \tag{$\star3$} p(s,t)\to 0\textup{ as }|s|\to\infty.
\end{align}
We also introduce the subspaces
\begin{align*}
    \M_{[0,R]}(\T^{2n})&:=\{(r,Z)\in\M\mid 0\leq r\leq R\textup{ and }\lim_{s\to\pm\infty}Z(s)\in\T^{2n}\times\{0\}\}\\
    \M_R(\T^{2n})&:=\{Z\mid (R,Z)\in\M_{[0,R]}(\T^{2n})\},
\end{align*}
where $\T^{2n}\times\{0\}\subseteq M$ denotes the set of constant maps. This coincides with the set of critical points of $\A_{H^0}$.  Note that for $R=0$ we get that $\M_0(\T^{2n})$ is equal to this set of critical points. 

It follows from \cref{thm:manifold} and \cref{thm:compactness} that $\M$ is a relatively compact finite-dimensional manifold, possibly after a small perturbation of $h$. Also the spaces $\M_{[0,R]}(\T^{2n})$ and $\M_R(\T^{2n})$ will then be compact manifolds (again after possibly perturbing $h$, see \cite{albers2016cuplength}). We have an evaluation map
\begin{align*}
    \textup{ev}_R:\M_R(\T^{2n})&\to (\T^{2n}\times B)^{2n}\\
    Z&\mapsto (Z(R,t_0),Z(2R,t_0),\ldots,Z(2nR, t_0)),
\end{align*}
where $t_0\in\T^2$ is some basepoint. For $R=0$ this gives the diagonal embedding $\T^{2n}\hookrightarrow (\T^{2n}\times B)^{2n}$.

The rest of the proof goes analogously to \cite{albers2016cuplength}. Take Morse functions $f_1,\ldots,f_{2n},f_*$ on $\T^{2n}$ and extend to Morse functions $\bar{f}_\alpha$ for $\alpha=1,\ldots,2n$ on $\T^{2n}\times B$ using a positive definite quadratic form on a tubular neighbourhood of $\T^{2n}$ in the direction normal to $\T^{2n}$. We also choose Riemannian metrics $g_1,\ldots,g_{2n},g_*$ on $\T^{2n}\times B$. For critical points $x_\alpha\in\T^{2n}$ of $f_\alpha$ and $x_*^\pm\in\T^{2n}$ of $f_*$ we define
\[
\M(R,x_1,\ldots,x_{2n},x_*^-,x_*^+):=\left\{Z\in\M_R(\T^{2n})\middle\vert (\star4)-(\star6)\right\},
\]
where
\begin{align}
    \tag{$\star4$} \textup{ev}_R&\in\prod_{\alpha}W^s(x_\alpha,\bar{f}_\alpha,g_\alpha)\\
    \tag{$\star5$} \lim_{s\to-\infty}Z(s)&\in W^u(x_*^-,f_*,g_*)\\ 
    \tag{$\star6$} \lim_{s\to+\infty}Z(s)&\in W^s(x_*^+,f_*,g_*).
\end{align}
Here, $W^u$ and $W^s$ denote the unstable and stable manifolds of the critical points for the gradient flows of the indicated Morse functions. Again, by choice of perturbations, we may assume these are all smooth manifolds. Note that $\M(0,x_1,\ldots,x_{2n},x_*^-,x_*^+)=\bigcap_{\alpha}W^s(x_\alpha,f_\alpha,g_\alpha)\cap W^u(x_*^-,f_*,g_*)\cap W^s(x_*^+,f_*,g_*)$.

Define the following operation on Morse co-chain groups
\begin{align*}
\theta_R&:\bigotimes_\alpha CM^*(f_\alpha)\otimes CM_*(f_*)\to CM_*(f_*)\\
x_1\otimes\cdots\otimes x_{2n}\otimes x_*^-&\mapsto \sum_{x^+_*\in\textup{Crit}(f_*)}\#_2\M(R,x_1,\ldots,x_{2n},x_*^-,x_*^+)\cdot x^+_*.
\end{align*}
Here, $\#_2$ denotes the parity of the set if it is discrete and 0 if it not. For generic choices of Morse functions and Riemannian metrics, the map $\theta_0$ realizes the usual cup-product in cohomology (see \cite{cupproduct}). This map does not vanish as the cuplength of $\T^{2n}$ is $2n$. Moreover, $\theta_R$ is chain homotopic to $\theta_0$ and can therefore not vanish either. Thus there have to be critical points $x_\alpha$ of $f_\alpha$ and $x_*^\pm$ of $f_*$ of index at least 1 such that the moduli spaces $\M(m,x_1,\ldots,x_{2n},x_*^-,x_*^+)$ are non-empty for all $m\in\N$. 

Let $Z_m\in\M(m,x_1,\ldots,x_{2n},x_*^-,x_*^+)$ and 
\[
Z^{(\alpha)}=\lim_{m\to\infty}Z_m(\cdot+m\alpha),
\]
for $\alpha=1,\ldots,2n$.
By the compactness discussed in \cref{sec:Compactness} these limits exist in $C^\infty_{loc}(\R,M)$. The $Z^{(\alpha)}$ satisfy the Floer \cref{eq:Floer} for the Hamiltonian $\tilde{H}^\rho$
and they converge to critical points $y^\pm_{\alpha}$ of $\A_{\tilde{H}^\rho}$ for $s\to\pm\infty$ such that
\[
\A_{\tilde{H}^\rho}(y_1^-)\geq \A_{\tilde{H}^\rho}(y_1^+)\geq \A_{\tilde{H}^\rho}(y_2^-)\geq \A_{\tilde{H}^\rho}(y_2^+)\geq\ldots\geq \A_{\tilde{H}^\rho}(y_{2n}^-)\geq \A_{\tilde{H}^\rho}(y_{2n}^+).
\] 
The only thing left to do is show that at least $2n+1$ of these points are different, by showing that the inequalities $\A_{\tilde{H}^\rho}(y_\alpha^-)>\A_{\tilde{H}^\rho}(y_\alpha^+)$ are strict. 

Note that since $\textup{Crit}(\A_{\tilde{H}^\rho})$ is finite by assumption, we may assume critical points of $\A_{\tilde{H}^\rho}$ do not intersect the stable manifolds $W^s(x_\alpha,\bar{f}_\alpha,g_\alpha)$. If $\A_{\tilde{H}^\rho}(y_\alpha^-)=\A_{\tilde{H}^\rho}(y_\alpha^+)$, then $Z^{(\alpha)}$ is constant and therefore a critical point of $\A_{\tilde{H}^\rho}$. Also note that by definition $Z_m(m\alpha)\in W^s(x_\alpha,\bar{f}_\alpha,g_\alpha)$ for any $m$. Therefore $Z^{(\alpha)}(0)=\lim_{m\to\infty}Z_m(m\alpha)$ lies in the closure of a stable manifold, which itself is a union of stable manifolds. This is in contradiction with $Z^{(\alpha)}$ being a critical point of $\A_{\tilde{H}^\rho}$. Thus indeed the inequalities $\A_{\tilde{H}^\rho}(y_\alpha^-)>\A_{\tilde{H}^\rho}(y_\alpha^+)$ are strict and we get that $y_1^-,y_2^-,\ldots,y_{2n}^-$ and $y_{2n}^+$ form $2n+1$ different critical points of $\A_{\tilde{H}^\rho}$. These correspond to $2n+1$ distinct solutions of \cref{eq:Ham} 
\end{proof}


\section{Fredholm theory and transversality}\label{sec:Fredholm}
In this section we will prove the following theorem.
\begin{theorem}\label{thm:manifold}
    The moduli space $\M$ from \cref{eq:ModuliSpace} is a finite-dimensional manifold for generic choice of $h$.
\end{theorem}
As is usual in Floer theory, this will be proven by describing $\M$ as the zero set of a Fredholm map. The results are standard, so we will not go into full detail. However, we would like to prove the core step, to stress that in our setting elementary Fourier arguments can be used. 

We let $\F$ denote the Floer map
\[\F(r,Z)=\F^r(Z)=\del_s Z +\delslash Z - \nabla\tilde{H}^{\rho,r}(Z).\]
This means that $\M=\F^{-1}(0)$, when $\F$ is viewed as a map on the right function spaces. We start by proving $\F$ is a nonlinear Fredholm map.

Let $r$ be fixed. Then the linearization of $\F^r$ at $Z$ is given by
\begin{align}\label{eq:linFloer}
(d\F^r)_Z = D-\beta_r(s)S_Z,
\end{align}
where $D=\del_s+\delslash-P$,
\begin{align*}
    P &= \begin{pmatrix}0&0&0&0\\0&0&0&0\\0&0&1&0\\0&0&0&1\end{pmatrix} &&\textup{and}&& S_Z=\textup{Hess}_Z\tilde{h}^\rho.
\end{align*}
By $H^k$ we denote the Sobolev space $W^{k,2}$. Note that for $k>\frac{3}{2}$ we have an embedding $H^k(\R\times\T^2,\R^{4n})\hookrightarrow C^0(\R\times\T^2,\R^{4n})$.
\begin{lemma}\label{lem:bijection}
    The operator $D:H^k(\R\times\T^2,\R^{4n})\to H^{k-1}(\R\times\T^2,\R^{4n})$ is a bijection for $k>\frac{3}{2}$.
\end{lemma}
\begin{proof}
    For $Y\in H^k:=H^k(\R\times\T^2,\R^{4n})$ consider the Fourier expansion
    \begin{align*}
        Y(s,t)&=\int\sum_{m_1,m_2\in\Z}\hat{Y}(\xi,m_1,m_2)e^{i\xi s}e^{im_1t_1}e^{im_2t_2}\,d\xi.
    \end{align*}
    Then on the Fourier coefficients the operator $D$ acts as 
    \begin{align*}
        \hat{D}(\xi,m_1,m_2) = i\xi + i\begin{pmatrix}0&0&-m_1&m_2\\0&0&-m_2&-m_1\\m_1&m_2&0&0\\-m_2&m_1&0&0\end{pmatrix}-P.
    \end{align*}
    We compute $\det(\hat{D}(\xi,m_1,m_2))=(m_1^2+m_2^2+\xi^2+i\xi)^2$. This means only constant functions can be in the kernel of $D$, however the only constant function that is square-integrable on $\R\times\T^2$ is the zero function. So $\ker D = 0$. 

    Now, for surjectivity let $W\in H^{k-1}$ and define 
    \begin{align*}
        \hat{Y}(\xi,m_1,m_2)=\begin{cases}\hat{D}(\xi,m_1,m_2)^{-1}\hat{W}(\xi,m_1,m_2) & (\xi,m_1,m_2)\neq (0,0,0)\\
        0 & (\xi,m_1,m_2)= (0,0,0).\end{cases}
    \end{align*}
    Note that by the calculation of the determinant above $\hat{D}(\xi,m_1,m_2)^{-1}$ exists. Now by definition we have that $DY=W$ provided $Y\in H^k$. To see this, note that the eigenvalues of $\hat{D}(\xi,m_1,m_2)$ are
    \[\lambda^\pm(\xi,m_1,m_2) = \frac{1}{2}i\left(i+2\xi\pm i\sqrt{1+4m_1^2+4m_2^2}\right).\]
    For $m_1^2+m_2^2$ large enough the norms of these eigenvalues can be bounded from below. In other words, there exists some $N$ such that for $m_1^2+m_2^2>N$, we have
    \begin{align*}
       |\lambda^\pm(\xi,m_1,m_2)|^2 \geq \xi^2+\frac{1}{2}m_1^2+\frac{1}{2}m_2^2.
    \end{align*}
    Let $\textup{pr}$ denote the projection onto the subspace where $m_1^2+m_2^2>N$. Then the $H^k$-norm of $\textup{pr}\,Y$ is
    \begin{align*}
        ||\textup{pr}\,Y||_{H^k}^2 &= \int\sum_{m_1^2+m_2^2>N}(1+\xi^2+m_1^2+m_2^2)^k|\hat{Y}(\xi,m_1,m_2)|^2\,d\xi\\
        &\leq \int\sum_{m_1^2+m_2^2>N}\frac{(1+\xi^2+m_1^2+m_2^2)^k}{\xi^2+\frac{1}{2}m_1^2+\frac{1}{2}m_2^2}|\hat{W}(\xi,m_1,m_2)|^2\,d\xi\\
        &\leq |||W||_{H^{k-1}}^2\sup_{\substack{m_1^2+m_2^2>N\\ \xi\in\R}}\frac{1+\xi^2+m_1^2+m_2^2}{\xi^2+\frac{1}{2}m_1^2+\frac{1}{2}m_2^2}<+\infty.
    \end{align*}
    Here, the third line follows from H\"olders inequality. Thus $Y\in H^k$, proving that $D$ is surjective. 
\end{proof}
\begin{corollary}\label{cor:floerfredholm}
    The Floer map $\F$ is a Fredholm map.
\end{corollary}
\begin{proof}
    Note that the operator $\beta_r(s)S_Z$ is compact as a map $H^k\to H^{k-1}$. Thus, by \cref{eq:linFloer} and \cref{lem:bijection} the differential of $\F^r$ and therefore of $\F$ is Fredholm.\footnote{Note that in our definition of $\M$ we require the $p$-coordinates to converge to zero, but we are not constraining the $q$-coordinates. For a rigorous setup of the function spaces in Morse-Bott scenarios see \cite{frauenfelder2004arnold}. The core step in proving Fredholmness still lies in \cref{lem:bijection}.}
\end{proof}

\subsection*{Transversality}
Left to prove is that for generic choice of non-linearity the differential of the Floer map is surjective. Combined with \cref{cor:floerfredholm} this yields \cref{thm:manifold}. 

Note that the moduli space $\M$ depends on the chosen non-linearity $\beta_r(s)\tilde{h}^\rho_t$. By varying these non-linearities for each $l$, we get the universal moduli space
\[\wt{\M}:=\bigcup_{h\in\S}\M(h_{s,t}),\]
where $\S=C_c^l(\R\times \T^2\times \T^{2n}\times B,\R)$ is the space of compactly supported smooth functions $h_{s,t}$ on $\T^{2n}\times B\subseteq T^*\T^{2n}$. We denote by $\M(h_{s,t})$ the moduli space $\M$ of \cref{eq:ModuliSpace}, where the non-linearity $\beta_r(s)\tilde{h}^\rho_t$ is replaced by $h_{s,t}$. Note that in this case $h_{s,t}$ is not necessarily the product of a function in $s$, with a time-dependent non-linearity on $\T^{2n}\times B$.

This universal moduli space can be seen as the zero set of the map
\[\wt{\F}(r,Z,h)=\wt{\F}^r(Z,h)=\del_s Z +\delslash Z - \nabla H^h_{s,t}(Z),\]
where $H^h_{s,t}(Z)=\frac{1}{2}|p|^2+h_{s,t}(Z)$ and $h$ vanishes for $s\notin [-1,(2n+1)r+1]$. The result would follow from Sard-Smale's theorem, if we know that the differential of $\wt{\F}^r$ is surjective at every $(Z,h)\in\wt{\M}$. Let $(Z,h)\in\wt{\M}$ and take a variation $(Y,G)\in H^k(\R\times \T^2,\R^{4n})\times C^l_c(\R\times\T^2\times\T^{2n}\times B,\R)$. Then the linearization of $\wt{\F}^r$ at $(Z,h)$ can be decomposed as 
\begin{align*}
    (d\wt{\F}^r)_{(Z,h)}\cdot(Y,G) = D^1_h G +D^2_Z Y.
\end{align*}
Assume we have some non-zero $W\in H^{k-1}(\R\times \T^2,\R^{4n})$ in the orthogonal complement of the image of $(d\wt{\F}^r)_{(Z,h)}$. Then $\langle W,D^2_Z Y\rangle=0$ for all $Y$, so that $W\in \ker \left((D^2_Z)^*\right)$. Note that for $s\notin[-1,(2n+1)r+1]$ the non-linearity $h_{s,t}$ vanishes so that $(D^2_Z)^*=-\del_s+\delslash-P$. This operator satisfies a maximum principle, similar to the one used in \cref{lem:Floerbounds}. Therefore, by standard arguments $W(s,t)\neq 0$ for some $(s,t)\in [-1,(2n+1)r+1]\times \T^2$. Thus, we can find some variation $G$ with support for $s\in[-1,(2n+1)r+1]$, such that $\langle W,D^1_h G\rangle \neq 0$. This contradicts the assumption that $\langle W,(d\wt{\F}^r)_{(Z,h)}\cdot(Y,G)\rangle=0$ for any $(Y,G)$. We must conclude $(d\wt{\F}^r)_{(Z,h)}$ is surjective. This finishes the proof of \cref{thm:manifold}.\qed

\section{Compactness}\label{sec:Compactness}
In this section we will prove compactness of the moduli space (\ref{eq:ModuliSpace}) that is needed in the proof of \cref{thm:cuplength}. We start by proving a uniform energy bound, using a standard argument.
\begin{lemma}\label{lem:Energybound}
    For all $(r,Z)\in\mathcal{M}$ we have 
    \begin{align*}
        E(Z)\leq 2||\tilde{h}^\rho||_{\textup{Hofer}}:=2\int_{\T^2}\left(\sup_{Z'\in T^*\T^{2n}}\tilde{h}^\rho_t(Z')-\inf_{Z'\in T^*\T^{2n}}\tilde{h}^\rho_t(Z')\right)\dV.
    \end{align*}
\end{lemma}
\begin{proof}
    By definition \[E(Z)=\int_{-\infty}^\infty\int_{\T^2}|\del_s Z(s,t)|^2\dV\,ds = -\int_{-\infty}^\infty\langle \grad\A_{\tilde{H}^{\rho,r}}(Z),\del_s Z\rangle\,ds,\] since $Z$ is a negative gradient line of $\A_{\tilde{H}^{\rho,r}}$. Writing this out further we get
    \begin{align*}
        E(Z) &= -\int_{-\infty}^\infty\left(d\A_{\tilde{H}^{\rho,r}}\right)_{Z_s}(\del_s Z)\,ds\\
        &=-\int_{-\infty}^\infty\frac{d}{ds}(\A_{\tilde{H}^{\rho,r}}(Z_s))\,ds+\int_{-\infty}^\infty\frac{\del\A_{\tilde{H}^{\rho,r}}}{\del s}(Z_s)\,ds\\
        &=-\lim_{s\to\infty}\A_{\tilde{H}^{\rho,r}}(Z_s)+\lim_{s\to-\infty}\A_{\tilde{H}^{\rho,r}}(Z_s)-\int_{-\infty}^\infty\int_{\T^2}\beta_r'(s)\tilde{h}^\rho_t(Z_s)\dV\,ds.
    \end{align*}
    Note that the first two terms on the last line vanish. Also, $\beta_r'$ vanishes everywhere outside $[-1,0]\cup[(2n+1)r,(2n+1)r+1]$ and is positive and negative respectively on those two intervals. Therefore,
    \begin{align*}
        E(Z) &=-\int_{-1}^0\int_{\T^2}\beta_r'(s)\tilde{h}^\rho_t(Z_s)\dV\,ds - \int_{(2n+1)r}^{(2n+1)r+1}\int_{\T^2}\beta_r'(s)\tilde{h}^\rho_t(Z_s)\dV\,ds\\
        &\leq 2\int_{\T^2}\left(\sup_{Z'\in T^*\T^{2n}}\tilde{h}^\rho_t(Z')-\inf_{Z'\in T^*\T^{2n}}\tilde{h}^\rho_t(Z')\right)\dV\\
        &=:2||\tilde{h}^\rho||_{\textup{Hofer}}.
    \end{align*}
\end{proof}

Now we want to relate the energy to the $L^2$-norm of the derivatives of $Z$. To this end we define the energy density
\[e_Z(s,t) = \frac{1}{2}|dZ(s,t)|^2 = \frac{1}{2}\left(|\del_s Z(s,t)|^2+|\del_1 Z(s,t)|^2+|\del_2 Z(s,t)|^2\right).\]
\begin{lemma}\label{lem:evsE}
    Let $K=[-\mu,\mu]\times\T^{2}\subseteq \R\times\T^{2}$. Then there exists a constant $C=C(K)$ such that for any $(r,Z)\in\M$ 
    \[\int_{K}e_Z(s,t)\dV\,ds \leq E(Z) + C(K).\]
\end{lemma}
\begin{proof}
    We calculate\footnote{The first equality follows from partial integration, using that $Z$ is nullhomotopic.}
    \begingroup
    \allowdisplaybreaks
    \begin{align*}
        \int_{K}e_Z(s,t)\dV\,ds &=\frac{1}{2}\int_K\left(|\del_s Z|^2+|2\del_t q|^2+|2\del_{\bar{t}} p|^2\right)\dV\,ds\\
        &=\frac{1}{2}\int_K\left(|\del_s Z|^2 +\left|\frac{\del \wt{H}^{\rho,r}}{\del p}-\del_s{p}\right|^2+\left|\del_s{q}-\frac{\del \wt{H}^{\rho,r}}{\del q}\right|^2\right)\dV\,ds\\
        &\leq E(Z) + \frac{1}{2}\int_K\left(\left|\frac{\del \wt{H}^{\rho,r}}{\del q}\right|^2+\left|\frac{\del \wt{H}^{\rho,r}}{\del p}\right|^2\right)\dV\,ds\\
        &\leq E(Z) + \frac{1}{2}\int_K\left(|p|^2+\left|\beta_r(s)\frac{\del \tilde{h}^\rho}{\del q}\right|^2+\left|\beta_r(s)\frac{\del\tilde{h}^\rho}{\del p}\right|^2\right)\dV\,ds.
    \end{align*}
    \endgroup
    Now, as in \cref{lem:Floerbounds} we have that $|p(s,t)|^2<\rho$ for all $(s,t)\in K$. Combined with the $C^1$-boundedness of $\tilde{h}^\rho$, we can estimate the right-hand side purely in terms of $E(Z)$ and the volume of $K$. This proves the lemma. 
\end{proof}

In order to get point-wise bounds on the derivatives, we want to use the Heinz trick (\cref{thm:Heinz}). We first need to show that the energy density satisfies an appropriate inequality.
\begin{lemma}\label{lem:densityineq}
    The energy density $e:=e_Z$ satisfies the inequality
    \begin{align}\label{eq:densityineq}
        (\del_s^2+\Delta)e\geq -c(1+e^{\frac{3}{2}}),
    \end{align}
    for some $c>0$.
\end{lemma}
\begin{proof}
    See \cref{app:energy}.
\end{proof}
\begin{theorem}[Heinz trick, see \cite{HNS}]\label{thm:Heinz}
    Let $M$ be an $(m+1)$-dimensional Riemannian manifold for $\frac{3}{2}<\frac{m+3}{m+1}$ and $e:M\to\R_{\geq0}$ a smooth function satisfying \cref{eq:densityineq}. Then for $K\subseteq M$ compact $\sup_K e$ is bounded by $\int_K e$.
\end{theorem}

Now we are ready to prove the desired compactness theorem.
\begin{theorem}\label{thm:compactness}
    The moduli space $\M$ is relatively compact in the $C^\infty_{\loc}$-topology.
\end{theorem}
\begin{proof}
    Let $(r_n,Z_n)\in\M$ for all $n\in\N$ and $K=[-\mu,\mu]\times\T^{2}\subseteq\R\times\T^2$. By \cref{lem:Floerbounds} the sequence $Z_n\vert_K$ is a bounded family of smooth functions on $K$. By combining \cref{lem:Energybound,lem:evsE} we find a uniform bound on $\int_K e_{Z_n}$. Then the Heinz trick (\cref{thm:Heinz}) combined with the density estimate in \cref{lem:densityineq} yields a uniform pointwise bound on $e_{Z_n}(s,t)$ for all $(s,t)\in K$. By the mean-value theorem it follows that the sequence $(Z_n)$ is uniformly equicontinuous. The theorem of Arzelà-Ascoli then yields a converging subsequence in the $C^0_{\loc}$-topology. Finally, the ellipticity of the Floer operator yields convergence of the subsequence in $C^\infty_{\loc}$.
\end{proof}

\appendix
\section{Proof of \cref{prop:metricsmooth}}\label{app:metric}
Let us start by proving the corresponding statement on linear space. 
\begin{lemma}
    Let $(V^{4n},I,\omega_1\otimes\del1+\omega_2\otimes\del_2)$ be a complex-regularized polysymplectic vector space, meaning that $\omega_1$ and $\omega_2$ are two linear symplectic forms on $V$ related by $\omega_2=-\omega_1(\cdot,I\cdot)$. There exists an inner product $g$ on $V$, such that 
    \begin{align*}
        \omega_1&=g(\cdot,J\cdot)\\
        \omega_2&=g(\cdot,K\cdot)
    \end{align*}
    for two anti-commuting linear complex structures $J,K$ satisfying $IJ=K$. 
\end{lemma}
\begin{proof}
    First we fix an auxiliary inner product $(\cdot,\cdot)$ on $V$ for which $I$ is an isometry and agree that transposes of linear maps on $V$ are taken with respect to this product. In particular $I^T=-I$. By non-degeneracy of $\omega_1$ and $\omega_2$, we find skew-symmetric isomorphisms $A_1,A_2:V\to V$ such that 
    \begin{align*}
        \omega_1&=(\cdot,A_1\cdot)\\
        \omega_2&=(\cdot,A_2\cdot).
    \end{align*}
    Note that for $X,Y\in V$
    \[(X,A_2Y)=\omega_2(X,Y)=-\omega_1(X,IY)=-(X,A_1IY),\]
    so that $A_2=-A_1I$. Also
    \begin{align*}
        (X,A_1IY) &= -\omega_2(X,Y)=\omega_2(Y,X)\\
        &=-\omega_1(Y,IX)=\omega_1(IX,Y)\\
        &=(IX,A_1Y)\\
        &=-(X,IA_1Y),
    \end{align*}
    which yields $A_1I=-IA_1$.

    To write down the polar decomposition of $A_1$, we define the positive-definite symmetric matrix $B=\sqrt{A_1A_1^T}$ and define $J$ by $A_1=BJ$. Note that 
    \[A_1A_1^T=(A_2I)(A_2I)^T=A_2II^TA_2^T=A_2A_2^T.\]
    Therefore the polar decomposition of $A_2$ is $A_2=BK$. Then both $A_1$ and $A_2$ commute with $B$ and therefore both $J$ and $K$ commute with $B$ as well. To prove that $J$ and $K$ are complex structures, we compute for $i=1,2$
    \[(B^{-1}A_i)^T=A_i^TB^{-1}=-A_iB^{-1}=-B^{-1}A_i\]
    and
    \[(B^{-1}A_i)(B^{-1}A_i)^T=B^{-1}A_iA_i^TB^{-1}=\Id_V.\]
    This shows that both $J$ and $K$ are skew-symmetric isometries of $(\cdot,\cdot)$, so that indeed
    \begin{align*}
        J^2=-JJ^T=-\Id_V && K^2=-KK^T=-\Id.
    \end{align*}
    Also, since $A_2=-A_1I$, we get 
    \[JI=-B^{-1}A_2=-K\]
    and 
    \[-\Id=K^2=JIJI,\]
    so that $K=-JI=IJ$. It follows that indeed $J$ and $K$ are anti-commuting linear complex structures. The metric $g$ may be defined as $g(X,Y)=(X,BY)$.
\end{proof}

\begin{proof}[Proof of \cref{prop:metricsmooth}]
    We may globally fix an auxiliary metric $(\cdot,\cdot)$ on $W$ for which $I$ is an isometry. The pointwise construction in the proof above is canonical after this choice and carries through smoothly, just like it does for the corresponding statement on symplectic manifolds.
\end{proof}

\section{Proof of \cref{lem:densityineq}}\label{app:energy}
For any $(r,Z)$ in $\M$, it holds that $Z$ satisfies the Floer equation
\begin{align}\label{eq:FloerJdel}
    \del_s Z +\delslash Z = \nabla \tilde{H}^{\rho,r}(Z).
\end{align}
For notational convenience we denote $s$ by $t_0$ and $\xi_i:=\del_iZ$ for $i=0,1,2$. From the identity 
\[(\del_s-\delslash)(\del_s+\delslash)=\L:=\sum_{i=0}^2\del_i^2\]
we see that the $\xi_i$ satisfy the equation
\begin{align}\label{eq:xiLaplace}
    \L\xi_i = (\del_s-\delslash)(\sigma_Z\xi_i),
\end{align}
where $\sigma_Z$ is the Hessian of $\tilde{H}^{\rho,r}$ at $Z$.

Note that in this notation the energy density is given by 
\[e(s,t) = \frac{1}{2}\sum_{i=0}^2|\xi_i(s,t)|^2.\]
We calculate
    \begin{align*}
        \L e &= \sum_{i,j=0}^2|\del_j\xi_i|^2 + \sum_{i=0}^2\langle \xi_i,\L \xi_i\rangle\\
        &=\sum_{i,j=0}^2|\del_j\xi_i|^2 + \sum_{i=0}^2\langle \xi_i,(\del_s-\delslash)(\sigma_Z\xi_i)\rangle,
    \end{align*}
where we have used \cref{eq:xiLaplace} in the second equality. Now note that the chain rule gives that
    \begin{align*}
        \del_j(\sigma_Z \xi_i)=T\sigma_Z\cdot \xi_j\cdot \xi_i+\sigma_Z\del_j \xi_i.
    \end{align*}
Denote $J_0=-\Id$, $J_1=J$ and $J_2=K$. Then
    \begin{align*}
        \L e&= \sum_{i,j=0}^2\left(|\del_j\xi_i|^2 -\langle (J_j \sigma_Z)^T\xi_i,\del_j\xi_i\rangle-\langle\xi_i,J_jT\sigma_Z\cdot\xi_j\cdot\xi_i\rangle\right)\\
        &=\sum_{i,j=0}^2\left(\left|\del_j\xi_i-\frac{1}{2}(J_j\sigma_Z)^T\xi_i\right|^2-\frac{1}{4}\left|(J_j\sigma_Z)^T\xi_i\right|^2-\langle\xi_i,J_jT\sigma_Z\cdot\xi_j\cdot\xi_i\rangle\right).
    \end{align*}
Note that since $\tilde{h}^\rho$ is $C^3$-bounded, we have that both $\sigma_Z$ and $T\sigma_Z$ are bounded. Thus
\[
    \left|(J_j\sigma_Z)^T\xi_i\right|^2\leq c_1|\xi_i|^2\leq c_2 e
\]
and
\[
    \left|\langle\xi_i,J_jT\sigma_Z\cdot\xi_j\cdot\xi_i\rangle\right|\leq c_3|\xi_i|^2|\xi_j|\leq c_4 e^{\frac{3}{2}}.
\]
Combined we get that
\[
\L e\geq -c_5\left(e+e^{\frac{3}{2}}\right).
\]
Finally, Young's inequality with $p=3$, $q=\frac{3}{2}$ gives that
    \[e\leq \frac{1}{3}+\frac{2}{3}e^{\frac{3}{2}},\]
    so that
    \[\L e\geq -c(1+e^{\frac{3}{2}}).\]
\qed

\bibliography{mybib}{}

\begin{thebibliography}{DLSsVf15}

\bibitem[AH16]{albers2016cuplength}
Peter Albers and Doris Hein.
\newblock {Cuplength estimates in Morse cohomology}.
\newblock {\em Journal of Topology and Analysis}, 8(02):243--272, 2016.

\bibitem[AM10]{cupproduct}
Peter Albers and Al~Momin.
\newblock {Cup-length estimates for leaf-wise intersections}.
\newblock In {\em Mathematical Proceedings of the Cambridge Philosophical
  Society}, volume 149, pages 539--551. Cambridge University Press, 2010.

\bibitem[Ass16]{ASSELLE20163513}
Luca Asselle.
\newblock {On the existence of Euler–Lagrange orbits satisfying the conormal
  boundary conditions}.
\newblock {\em Journal of Functional Analysis}, 271(12):3513--3553, 2016.

\bibitem[BD99]{bridgesderks}
Thomas~J Bridges and Gianne Derks.
\newblock {Unstable eigenvalues and the linearization about solitary waves and
  fronts with symmetry}.
\newblock {\em Proceedings of the Royal Society of London. Series A:
  Mathematical, Physical and Engineering Sciences}, 455(1987):2427--2469, 1999.

\bibitem[BF23]{regularizedpolysympl}
Ronen Brilleslijper and Oliver Fabert.
\newblock {Regularized polysymplectic geometry and first steps towards Floer
  theory for covariant field theories}.
\newblock {\em Journal of Geometry and Physics}, 183:104703, 2023.

\bibitem[BF24]{hyperkahlerrookworst}
Ronen Brilleslijper and Oliver Fabert.
\newblock {From Euclidean field theory to hyperkahler Floer theory via
  regularized polysymplectic geometry}.
\newblock {\em Communications in Contemporary Mathematics}, 2024.

\bibitem[Bri06]{bridgesTEA}
Thomas~J Bridges.
\newblock {Canonical multi-symplectic structure on the total exterior algebra
  bundle}.
\newblock {\em Proceedings of the Royal Society A: Mathematical, Physical and
  Engineering Sciences}, 462(2069):1531--1551, 2006.

\bibitem[DLSsVf15]{deLeonMethods}
Manuel De~Le{\'o}n, Modesto Salgado-seco, and Silvia Vilarino-fernandez.
\newblock {\em {Methods of differential geometry in classical field theories:
  k-symplectic and k-cosymplectic approaches}}.
\newblock World Scientific, 2015.

\bibitem[DR22]{doan2022holomorphic}
Aleksander Doan and Semon Rezchikov.
\newblock {Holomorphic Floer theory and the Fueter equation}.
\newblock {\em arXiv preprint arXiv:2210.12047}, 2022.

\bibitem[Fra04]{frauenfelder2004arnold}
Urs Frauenfelder.
\newblock {The Arnold-Givental conjecture and moment Floer homology}.
\newblock {\em International Mathematics Research Notices},
  2004(42):2179--2269, 2004.

\bibitem[GH12]{ginzburg2012hyperkahler}
Viktor~L Ginzburg and Doris Hein.
\newblock {Hyperk{\"a}hler Arnold conjecture and its generalizations}.
\newblock {\em International Journal of Mathematics}, 23(08):1250077, 2012.

\bibitem[GH13]{ginzburg2013arnold}
Viktor~L Ginzburg and Doris Hein.
\newblock {The Arnold conjecture for Clifford symplectic pencils}.
\newblock {\em Israel Journal of Mathematics}, 196:95--112, 2013.

\bibitem[G{\"u}n87]{gunther1987polysymplectic}
Christian G{\"u}nther.
\newblock {The polysymplectic Hamiltonian formalism in field theory and
  calculus of variations. I. The local case}.
\newblock {\em Journal of differential geometry}, 25(1):23--53, 1987.

\bibitem[HNS09]{HNS}
Sonja Hohloch, Gregor Noetzel, and Dietmar~A Salamon.
\newblock {Hypercontact structures and Floer homology}.
\newblock {\em Geometry \& Topology}, 13(5):2543--2617, 2009.

\bibitem[KS24]{kontsevich2024holomorphic}
Maxim Kontsevich and Yan Soibelman.
\newblock {Holomorphic Floer theory I: exponential integrals in finite and
  infinite dimensions}.
\newblock {\em arXiv preprint arXiv:2402.07343}, 2024.

\bibitem[Lat14]{latschev2014fukaya}
Janko Latschev.
\newblock {Fukaya's work on Lagrangian embeddings}.
\newblock {\em arXiv preprint arXiv:1409.6474}, 2014.

\bibitem[MS17]{mcduff2017introduction}
Dusa McDuff and Dietmar Salamon.
\newblock {\em {Introduction to symplectic topology}}, volume~27.
\newblock Oxford University Press, 2017.

\bibitem[Oh02]{oh2002holomorphic}
Yong-Geun Oh.
\newblock {Holomorphic volume preserving maps and special Lagrangian
  submanifolds}.
\newblock {\em Contemporary mathematics}, 314:199--208, 2002.

\bibitem[Sch98]{schwarz1998quantum}
Matthias Schwarz.
\newblock {A quantum cup-length estimate for symplectic fixed points}.
\newblock {\em Inventiones mathematicae}, 133(2):353--397, 1998.

\bibitem[SS03]{stein2003complex}
Elias~M Stein and Rami Shakarchi.
\newblock {\em Complex analysis}, volume~2.
\newblock Princeton University Press, 2003.

\bibitem[Wag23]{wagner2023pseudo}
Luiz~Frederic Wagner.
\newblock {Pseudo-Holomorphic Hamiltonian Systems}.
\newblock {\em arXiv preprint 2303.09480}, 2023.

\bibitem[Wal17]{walpuski2017compactness}
Thomas Walpuski.
\newblock {A compactness theorem for Fueter sections}.
\newblock {\em Commentarii Mathematici Helvetici}, 92(4):751--776, 2017.

\end{thebibliography}
\bibliographystyle{alpha}

\end{document}